\documentclass[12pt,a4paper]{amsart}
\usepackage{latexsym,amsfonts,amsmath,amssymb,amsthm,rotating,txfonts}
\usepackage{tikz,pdfsync}
\newtheorem{theorem}{Theorem}
\newtheorem{proposition}[theorem]{Proposition}
\newtheorem{lemma}[theorem]{Lemma}
\newtheorem{corollary}[theorem]{Corollary}
\theoremstyle{definition}
\newtheorem{definition}[theorem]{Definition}
\newtheorem{example}[theorem]{Example}
\theoremstyle{remark}
\newtheorem{remark}[theorem]{Remark}

\newcommand{\T}{T}
\newcommand{\TT}{\mathrm{T}}
\newcommand{\TM}{\TT M}
\renewcommand{\t}{{\mathfrak t}}

\newcommand{\q}{{\mathfrak q}}
\renewcommand{\a}{{\mathfrak a}}
\renewcommand{\b}{{\mathfrak b}}
\newcommand{\g}{{\mathfrak g}}

\newcommand{\rz}{{\t_{\ZZ}^\perp}}
\newcommand{\rzm}{{\t_M^\perp}}
\renewcommand{\r}{S}
\newcommand{\ZZ}{C}

\newcommand{\n}{{\mathfrak n}}
\newcommand{\C}{{\mathbb C}}
\renewcommand{\P}{{\mathbb P}}

\newcommand{\R}{{\mathbb R}}
\newcommand{\Z}{{\mathbb Z}}
\newcommand{\CA}{\mathcal{ A}}

\newcommand{\CB}{\mathcal{ B}}
\newcommand{\CE}{\mathcal{ E}}
\newcommand{\CEk}{\CE\tensor\CL^k}

\newcommand{\CL}{\mathcal{ L}}

\newcommand{\CS}{{S}}
\newcommand{\CR}{\mathcal{ R}}

\newcommand{\Todd}{\operatorname{Todd}}

\newcommand{\Tr}{\operatorname{Tr}}
\newcommand{\ch}{\operatorname{ch}}

\newcommand{\Ker}{\mathrm{Ker}}
\newcommand{\Coker}{\mathrm{Coker}}

\newcommand{\rank}{\operatorname{rank}}

\newcommand{\dom}{{\operatorname{dom}}}
\newcommand{\even}{{\operatorname{even}}}
\newcommand{\odd}{{\operatorname{odd}}}

\newcommand{\Sub}{\FPC}
\newcommand{\loc}{\Delta}

\newcommand{\chie}{\chi_\CE}
\newcommand{\chiek}{\chi_{\CEk}}

\newcommand{\End}{\mathrm{End}}
\newcommand{\F}{\mathcal{F}\!}

\newcommand{\rhat}{\hat R(\T)}
\newcommand{\Xpol}{Y}

\newcommand{\ccp}{\tau_p[\CE]}
\newcommand{\ccpb}{\tau_p[\CE^\bullet]}

\newcommand{\CRF}{\CR(\Phi)}
\newcommand{\tensor}{\otimes}
\newcommand{\vv}{\mathcal{T}}

\newcommand{\setm}{\!\setminus\!}

\newcommand{\felarrow}{\!\uparrow\!}

\newcommand{\yv}{\ysg}
\newcommand{\gammap}{\gamma-\mu(p)}
\newcommand{\ap}{\a-\mu(p)}

\newcommand{\qpol}{\delta}
\newcommand{\roots}{\mathfrak{R}}
\newcommand{\muperp}{\mu_\perp}

\newcommand{\reg}{{\mathrm{reg}}}
\newcommand{\unit}{\mathbf{1}}
\newcommand{\WT}{\wedge^\bullet \TT^*M}
\newcommand{\yzz}{\yc}
\newcommand{\NZZ}{KC}

\newcommand{\ysg}{Y_{S,\gamma}}
\newcommand{\am}{A_M}
\newcommand{\yc}{Y_C}
\newcommand{\tjs}{{\bar \TT}^J}
\newcommand{\tj}{{\TT}^J}
\newcommand{\FPC}{\mathrm{Comp}_T}

\newcommand{\lspan}{\mathrm{span}}
\newcommand{\termec}{\mathrm{Term}_C[\mu,\CE,\gamma]}
\newcommand{\mcont}{\mathrm{Term}}
\newcommand{\G}{\mathcal{G}}
\newcommand{\bba}{\tilde{\mathfrak{b}}}

\title{$[Q,R]=0$ and Kostant partition functions } \author{A. Szenes}
\thanks{The support of FNS grants 132873 and 126817 is gratefully
  acknowledged.}  \author{M. Vergne}

\begin{document}

\maketitle
\tableofcontents

\section{Introduction}
\label{sec:thm}

\subsection{Quantization and multiplicities}
Let $M$ be a compact almost complex manifold. The complex structure
$J\in\Gamma\End(\TM)$ then induces the splitting $\TM\tensor \C=\tj
M\oplus \tjs M$, where $\tj M$ is the complex vector bundle of
$+i$-eigenspaces, while $\tjs M$ is the bundle of $-i$-eigenspaces of
$J$ acting on $\TM\tensor\C$.  When $M$ is a complex manifold endowed
with an Hermitian metric, then $\tj M$ may be identified with the
  complex tangent bundle, while $\tjs M$ with the complex
  cotangent bundle of $M$.

To every complex vector bundle $\CE\to M$ over $M$ one can associate
an integer as follows (see \eqref{qdim} below).  Set the notation
$\Omega_J^{\bullet}(M,\CE)=\Gamma(\wedge^{\bullet} (\tjs M)^*\otimes
\CE)$ for the anti-holomorphic differential forms with values in
$\CE$, and consider the twisted Dolbeault-Dirac operator
\cite{ber-get-ver}
$$D_{\CE}:\Omega_J^{\even}(M,\CE)\to \Omega_J^{\odd}(M,\CE),$$
which is a first-order elliptic differential operator on $M$.
We can associate to this operator the $\Z_2$-graded
vector space
\begin{equation}
  \label{qdef}
 Q(M,\CE)=\Ker(D_\CE)\oplus\Coker (D_\CE),
\end{equation}
where  $\Ker(D_\CE)$ is placed in the even part, while $\Coker
(D_\CE)$ in the odd part of $Q(M,\CE)$.

\begin{remark}
Thought of as the formal
difference of $\Ker(D_\CE)$ and $\Coker (D_\CE)$ one can think of
$Q(M,\CE)$ as the virtual space of solutions of the corresponding
differential equations.
\end{remark}

The (super)-dimension of this $\Z_2$-graded vector space is defined to
be the integer
\begin{equation}
  \label{qdim}
  \dim Q(M,\CE)=\dim \Ker(D_\CE)-\dim \Coker (D_\CE).
\end{equation}        This number may be computed
by the Atiyah-Segal-Singer index formula:
\begin{equation}
  \label{atiyahsinger}
  \dim Q(M,\CE)=\int_M\ch(\CE)\Todd(\tj M);
\end{equation}
here $\ch(\CE)$ is the Chern character of $\CE$ and $\Todd(\tj M)$ is
the Todd class of $M$.

Now assume that a compact, connected Lie group $G$ acts compatibly on
the manifold $M$ and the bundle $\CE$, and preserves the almost
complex structure $J$.  Then $Q(M,\CE)$ becomes a $\Z_2$-graded
representation of $G$, and we still denote by $Q(M,\CE)$ the
corresponding element $\Ker(D_\CE)-\Coker (D_\CE)$ of the Grothendieck
ring $R(G)$ of virtual representations of $G$.  We will be interested
in the decomposition of this virtual representation into irreducible
components.

To make this more explicit, we introduce the following notation for the Lie data:
\begin{itemize}
\item  Denote by $\T$ the maximal torus of $G$, and
\item by $\g$ and $\t$ the Lie algebras of $G$ and $\T$, respectively;
\item we will identify $\t^*$ with the $\T$-invariant subspace of
  $\g^*$ under the coadjoint action.
  \item Let  $\Lambda$ stand for the weight lattice of $\T$ thought of as a
  subspace of $\t^*$.
\item We will use the notation $e_{\lambda}$ for the
character $\T\to\C^*$  corresponding to $\lambda\in\Lambda$, and write
$t^\lambda$ for the value of this character on $t\in\T$. Thus we have
$e_\lambda(t) = t^\lambda$ for $t\in \T$, and also
$t^\lambda=e^{i\langle\lambda,X\rangle}$ if $X\in\t$ and $t=\exp(X)$.

\item Denote the set of roots of $G$ by $\roots$, and choose a
  splitting of $\roots$ into a positive and a negative part:
  $\roots=\roots^+\cup\roots^-$. Let $\g_\C=\t_{\C}\oplus \n^+\oplus
  \n^-$ be the corresponding triangular decomposition of the
  complexification of the Lie algebra $\g$ of $G$.

\item Write $dt$ for the Haar measure on $\T$ satisfying $\int_\T\,dt=1$.
\end{itemize}

Further, we introduce the following notation:
\begin{itemize}
\item for $X\in\g$, we denote by $VX$ the vector field
\[ VX:M\to \TM, \quad VX: \,q\mapsto\frac{d}{dt} e^{-t X} q|_{t=0}\] on
$M$ induced by the $G$-action;
\item we define the character $\chie: T\to\C$ via
\[   \chie(t)  = \Tr t |\, \Ker(D_\CE) - \Tr t|\,\Coker (D_\CE).
\]
\end{itemize}
Atiyah-Bott-Segal-Singer \cite{ati-bot2, ati-bot3,Atiyah-Segal68} gave
a formula for $\chie(t)$ in terms of the connected components of the
set of fixed points of the action of $t$ on $M$.  The Fourier
transform $\F\chie:\Lambda\to\Z$ of $\chie$ is a function with finite
support; its value \[ \F\chie(\lambda) = \int_\T
t^{-\lambda}\,\chi_{\CE}(t)\,dt \] is an integer, called the {\em
  multiplicity} of the weight $\lambda$ in $\chie$.  Using the fixed
point formula, one can express $\F\chie(\lambda)$ in terms of
partition functions. The first example of such an expression was
Kostant's formula for the multiplicity of a weight in a
finite-dimensional representation of a compact Lie group in terms of
the number of ways a weight can be expressed as a sum of positive
roots.

Our focus will be  the calculation of the dimension of the
$G$-invariant part $Q(M,\CE)^G$   of $Q(M,\CE)$, obtained by
taking $G$-invariants on the right hand side of \eqref{qdef}. Thus we have
\[   \dim Q(M,\CE)^G = \dim \Ker(D_\CE)^G-\dim\Coker (D_\CE)^G.
\]
According to the Weyl character formula, this integer may be
expressed via the multiplicities using the formula
\begin{equation}
  \label{reduct}
  \dim Q(M,\CE)^G
=\int_\T\prod_{\alpha\in\roots^-}(1-t^{\alpha})\,\chie(t)\,dt.
\end{equation}

A key tool of our approach is a formula of Paradan \cite{par2},
expressing $\chie(t)$ as a sum of characters of infinite dimensional
virtual representations of $T$ associated to a collection of subtori
of $T$ (cf. Proposition \ref{paradan}).  We will give a direct proof
of this result, deriving it from the Atiyah-Bott-Segal-Singer 

Let us demonstrate Paradan's formula for $\chie(t)$ is the
simplest example: that of the complex projective line.

\begin{example}
\label{ponex1}
Let  $M=\P^1(\C)$ be
endowed with the action of the group $G=\mathrm{SU}(2)$.
Let $\CL$ be the dual of the tautological bundle, and let $\CE=\CL^k$.
The maximal
torus $\T$ of the group $G$ corresponds to the set of diagonal matrices
in $\mathrm{SU}(2)$.  The action of $t\in T$ on $P^1(\C)$ is given by
$t\cdot(x:y)=(tx:t^{-1}y)$.  The Atiyah-Bott formula reads
as
\begin{equation}
  \label{ab22}
\chi_{\CL^k}(t)=\frac{t^k}{1-t^{-2}} +\frac{t^{-k}}{1-t^2}.
\end{equation}
Then $\chi_{\CL^k}(t) =
 \begin{cases}\sum_{j=0}^kt^{k-2j},\hspace{1.9cm}\text{if }\,0\leq k,\\
0,\;\hspace{3cm}\text{if } k=-1,\\
-\sum_{j=0}^{-k-2}t^{-k-2-2j},\;\hspace{0.8cm}\text{if }\,k< 0.\\
  \end{cases}$

The dimension of the virtual representation  $Q(M,\CL^k)$ is equal to
$\chi_{\CL^k}(1)$, which is equal to $k+1$ in our case.

Expanding $(1-t^2)^{-1}$ as the geometric series $\sum_{j=0}^{\infty} t^{2j}$, we obtain
\begin{equation}
  \label{exp2}
\chi_{\CL^k}(t)=\sum_{j=0}^{\infty}t^{-k+2j}- \sum_{j=1}^{\infty} t^{k+2j}.
\end{equation}

Using the identity $\sum_{j=0}^{\infty}
t^{2j}=\sum_{j=-\infty}^{\infty} t^{2j} -\sum_{j=-\infty}^{-1}t^{2j},$
we obtain Paradan's symmetric expression for $\chi_{\CL^k}(t)$, which is the
sum of three formal characters

\begin{equation}
  \label{expone}
\chi_{\CL^k}(t)=t^k\sum_{j=-\infty}^{\infty}t^{2j}- \sum_{j=1}^{\infty}t^{-k-2j}- \sum_{j=1}^{\infty}t^{k+2j}.
\end{equation}

The character $t^k\sum_{j=-\infty}^{\infty}t^{2j}$ depends on
$k\,\mathrm{mod}\,2$ only; it defines a generalized function on $T$
supported at $t=\pm 1$, and hence it is ``invisible'' at any $t\neq \pm 1$.
\end{example}

\subsection{Quantization of symplectic manifolds}
Consider an equivariant line bundle $\CL$ over $M$, endowed with a
$G$-invariant Hermitian structure and an Hermitian connection
$\nabla$. Then the curvature $\nabla^2$ will be of the form
$-i\Omega$, where $\Omega$ is a closed real 2-form on $M$.  The
$G$-invariant connection $\nabla$ determines a $G$-equivariant map
$\mu_G: M\to\g^*$, called the {\em moment map}:
\begin{equation}
  \label{defmu}
i\langle\mu_G,X\rangle=L_X-\nabla_{VX},
\end{equation}
where $L_X$ is the Lie derivative acting on the sections of $\CL$.
Observe that if $p\in M$ is a fixed point of the $\T$-action, then
$\mu_G(p)$ is in $\t^*\subset\g^*$, moreover, $\mu_G(p)$ is exactly
the $\T$-weight of the fiber $\CL_p$.  Differentiating \eqref{defmu},
we obtain the key identity
\begin{equation}
  \label{musat}
\langle
d\mu_G,X\rangle+\Omega(VX,\cdot) =0.
\end{equation}

The goal of this article is to give new proofs of certain
polynomiality properties of the function $k\mapsto \dim Q(M,\CL^k)^G$.

First, consider the case where $G=\T$ is abelian. In this case, we
will write $\mu:M\to\t^*$ for the moment map, omitting the index $\T$.
Our first result concerns the case of large $k$.
\begin{theorem} \label{first} Let $\CE^\even$ and $\CE^\odd$ be
  $\T$-equivariant vector bundles over the almost complex manifold
  $M$.  Let $\CL$ be an equivariant line bundle with associated moment
  map $\mu:M\to\t^*$. Suppose that $\CE^\even$ and $\CE^\odd$ restricted to
  $\mu^{-1}(0)$ are isomorphic as $T$-equivariant vector bundles.
  Then, for $k$ large, the multiplicities
  $\F\chi_{\CE^\even\tensor\CL^k}(0)$ and $\F\chi_{\CE^\odd\tensor\CL^k}(0)$
  are equal.  \end{theorem}
We give a proof of this theorem in \S\ref{sec:appendix}, following
Meinrenken, based on the stationary phase principle applied to the
integral formula of \cite{ber-ver85} for $\chiek$.

Now we turn to the case of a general compact connected $G$. We will need
to weaken the notion of polynomiality as follows.
\begin{definition}\label{defquasi}
  Let $\Xi$ be a  lattice, i.e. a free  $\Z$-module of finite rank. A function
  $P:\Xi\to\C$ is {\em quasi-polynomial} if there exists a
  sublattice $\Xi_0\subset\Xi$ of finite index such that for
  every $\lambda\in\Xi$ the function $P$ restricted to
  $\lambda+\Xi_0$ coincides with the restriction of a polynomial
  function from $\Xi$ to $\Xi_0$.
\end{definition}
In particular, a function $P:\Z\to\C$ is quasi-polynomial if, for some
nonzero $d\in\Z$, the function $l\mapsto P(ld+r)$ is polynomial for
every $r\in\Z$.

\bigskip

\begin{example}
\label{ponex}
We return to Example \ref{ponex1}. We compiled the relevant data in the following table:\\
  \begin{center}
\begin{tabular}{ |l || l | r | r | r | r | r | r | r | r | r| }
\hline
$k$ &  \dots &  $-4$ & $-3$ & $-2$ & $-1$ & $0$ & $1$ & $2$ & $3$  &  \dots\\ \hline
$\dim Q(M,\CL^k)  $ & \dots & $-3$ & $ -2 $ & $ -1 $ & $ 0 $ & $ 1 $ & $ 2 $ & $ 3 $ & $ 4 $ &  \dots\\ \hline
$\dim Q(M,\CL^k)^{\T}  $ & \dots & $-1$ & $ 0 $ & $ -1 $ & $ 0 $ & $ 1 $ & $ 0 $ & $ 1 $ & $ 0 $ &  \dots\\ \hline
$\dim Q(M,\CL^k)^{\mathrm{SU}(2)}  $ & \dots & $0$ & $ 0 $ & $ -1 $ & $ 0 $ & $ 1 $ & $ 0 $ & $ 0 $ & $ 0 $ &  \dots\\ \hline
 \end{tabular}
  \end{center}
Thus we see that
\begin{itemize}
\item $\dim Q(M,\CL^k) = k+1$; it is thus a polynomial for all $k\in\Z$.
\item $\dim Q(M,\CL^k)^{\T} =
  \begin{cases}
    1,\;\text{if }\,0\leq k\text{ is even},\\
- 1,\;\text{if }0>k\text{ is even},\\
0,\;\text{if }k\text{ is odd}.
  \end{cases}
$ \\
In particular, this is a quasi-polynomial for all $k\geq0$.
\item $\dim Q(M,\CL^k)^{\mathrm{SU}(2)} $ is, however, only quasi-polynomial
  for $k\geq1$, and $\dim Q(M,(\CL^{-1})^k)^{\mathrm{SU}(2)} $ is {\em
    not} quasi-polynomial for $k\geq1$.
\end{itemize}

\end{example}

This last example shows, that, in general, $\dim Q(M, \CL^k)^G$ is not
quasi-polynomial for small $k$.  To obtain a stronger statement, we
introduce a key condition on $\CL$.

\begin{definition}  \label{defpositive}
  Given an almost complex manifold $(M,J)$, we say that a line bundle
  $\CL$ over $M$ is {\em positive} if for an Hermitian structure on $\CL$, and a
  compatible connection $\nabla$, the corresponding curvature
  $-i\Omega$ satisfies
  \begin{equation}
    \label{poscond}
      \Omega_q(V,JV)>0\quad \text{for all }0\neq V\in \T_qM
  \end{equation}
at every point $q\in M$.
\end{definition}
\begin{remark}
Note that in this case, $\Omega$ is a symplectic form on $M$.
\end{remark}

One can arrive at the same setup starting at the other end: let
$(M,\Omega)$ be a symplectic manifold endowed with a line bundle $\CL$,
whose curvature is $-i\Omega$. Such an object is called a
prequantizable symplectic manifold endowed with a Kostant line bundle
\cite{kos}. In this case, one can choose a unique (up to homotopy)
almost complex structure $J$ such that the quadratic form
$V\mapsto\Omega_q(V,JV)$ is positive definite at each point $q\in M$,
and thus one arrives at the situation described in Definition
\ref{defpositive}.  In addition, if such a Kostant line bundle $\CL$ is
endowed with a $G$-action and a $G$-invariant connection, then the
virtual representation space $Q(M,\CL)$ does not depend on the choice of such
a (positive) $G$-invariant almost complex structure $J$.

Now we are ready to formulate the statement for which we
give a new proof in this article.
(As we explain below, this theorem may be obtained as a corollary of results
of \cite{mei-sja}.)
\begin{theorem}\label{GS}
  Let $(M,J)$ be a compact, connected, almost complex manifold endowed
  with the action of a connected compact Lie group $G$, and let $\CL$
  be a positive $G$-equivariant line bundle on $M$. Assume that the set of
  fixed points under the action of the maximal torus $T$ of $G$ on $M$
  is finite.  Then
  \begin{itemize}
  \item the integer function
\[k\to \dim   Q(M,\CL^k)^G
\] is quasi-polynomial  for $k\geq 1$, and
\item this quasi-polynomial is identically zero if $0\notin \mu_G(M)$.
  \end{itemize}
\end{theorem}

\begin{remark}
  Note that the condition of the finiteness of the $T$-fixed point set
  is not necessary. We chose to impose this condition solely to
  simplify the discussion. To prove the theorem in the case of
  non-isolated fixed points, one  needs to use the equivariant
  index formula of Atiyah-Segal-Singer \cite{Atiyah-Segal68}, instead
  of the Atiyah-Bott fixed point formula \cite{ati-bot2}.
\end{remark}

\subsection{The ideas of the proof} At first sight, the strategy seems
to be clear. The Atiyah-Bott formula gives an explicit formula for
$\chi_{\CL^k}$ as a sum of rational functions (cf.
\eqref{ab}). Choosing a generic direction in $\t$, we can expand these
rational functions into convergent series, obtaining a formula of the
form $\chi_{\CL^k}=\sum_{p\in F} e_{k\mu_p}\theta_p$, where $\theta_p$
is a formal character, whose coefficients are given by a partition
function, and $\mu_p$ is the weight of $\CL_p$ (cf. \eqref{exp2}). To
obtain a formula for $\dim Q(M,\CL^k)^G$ when $G$ is a torus group,
one simply needs to evaluate the constant term of this expansion. This
leads to a formula of the form
\[      \dim Q(M,\CL^k)^G = \sum_{p\in F} \F\theta_p(-k\mu_p),
\]
where $\F\theta_p(\lambda)$ stands for the multiplicity of $e_\lambda$
in $\theta_p$.
The contribution of each fixed point to this constant term is a
polynomial in $k$, and thus, in this case, the proof of polynomiality is
straightforward. 

When $G$ is a general connected compact group, then
we need to use  \eqref{reduct}, and we obtain a formula of the
form
\[      \dim Q(M,\CL^k)^G = \sum_{p\in F}
\sum_{J\subset\roots^{-}}(-1)^{|J|}\F\theta_p\left(-\sum_{\alpha\in J}\alpha-k\mu_p\right).
\] Here, because of the shifts by sums of negative roots, the
individual terms are no longer polynomial for small values of $k$, and
polynomiality is the result of a complicated web of cancelations.

The novel idea of Paradan, which goes back to the seminal paper of
Witten \cite{wit92}, is to use a certain combinatorial expansion of
the rational functions from the Atiyah-Bott fixed point formula, which
has terms expanded in different directions, always away from the
origin (cf. \eqref{expone}).  After resummation, one obtains a formula
(Proposition \ref{resumprop}), whose terms are parametrized by fixed
point sets of subtori of the maximal torus $T\subset G$. Finally, we
show that the polynomiality of $\dim Q(M,\CL^k)^G$ hinges on a
geometric statement about the weights of the action of these subtori
on the tangent space of $M$ (Proposition \ref{crucial}).

\subsection{Comments on [Q,R]=0 and polynomiality}
\label{sec:rempol}
Quantization commutes with reduction (or [Q,R]=0 for short) is the
principle that the virtual space $Q(M,\CE\otimes \CL^k)^G$ may be
identified with the virtual space of solutions of a Dirac operator
associated to a vector bundle of the form $\CE_0\otimes\CL_0^k$ on the
so-called {\em reduced space} $\mu_G^{-1}(0)/G$.  If this latter space
is smooth, then, using this principle and applying the Atiyah-Singer
formula \eqref{atiyahsinger} to the bundle $\CL_0^k$, we can conclude
that $\dim Q(M,\CL^k)^G$ depends polynomially on $k$. The
polynomiality  of this dimension function thus is a key manifestation of
the [Q,R]=0 principle.

The idea of [Q,R]=0 was introduced in \cite{Gui-ste} (cf.  \cite{sja}
and \cite{ver01} for more details and references) in the form of a
precise conjecture. The idea came from considering the case when $M$
is a complex projective $G$-manifold, $\CL$ is the ample bundle and
$\CE$ is trivial.  Then the $G$-action on $M$ may be extended to a
holomorphic action $G_\C\times M\to M$ of the complexification of the
compact Lie group $G$, and [Q,R]=0 follows from the fact that
(cf. \cite {mum-fog-kir}) the orbit of the set $\mu_G^{-1}(0)$ under
this complexified action of $G_\C$ is dense in $M$ if this orbit is nonempty.

If $0$ is a regular value of $\mu_G$, then the reduced space
$\mu_G^{-1}(0)/G$ is a symplectic orbifold equipped with a Kostant
line bundle $\CL_0$. Guillemin-Sternberg formulated the conjecture
that $Q(M,\CL)^G$ may be identified with $Q(\mu_G^{-1}(0)/G,\CL_0)$.

Meinrenken, in his first approach to the Guillemin-Sternberg
conjecture \cite{mei1}, determined the asymptotic behavior of $\dim
Q(M,\CL^k)^G$ for large $k$, under the assumption that $0$ is a
regular value of $\mu_G$. By a ``stationary phase" argument (that we
borrowed in part for our proof of Theorem \ref{first}), he showed that
$\dim Q(M,\CL^k)^G$ is indeed equal to $\dim
Q(\mu_G^{-1}(0)/G,\CL_0^k)$ for $k$ sufficiently large, and that the
equality holds for all $k\geq0$ if $G$ is abelian.  His results thus
imply that the Guillemin-Sternberg conjecture for general compact
connected Lie group $G$ is equivalent to the fact that $\dim
Q(M,\CL^k)^G$ is quasi-polynomial in $k$ for $k\geq 1$.

Meinrenken-Sjamaar in \cite{mei-sja} formulated the
Guillemin-Sternberg conjecture for the case when 0 is not necessarily a
regular value of the moment map, and, using techniques of symplectic
cutting, proved this more general statement.
 There
is also an analytic proof of this generalized Guillemin-Sternberg
conjecture by Tian and Zhang \cite{tia-zha}, and, later,  another proof by
Paradan \cite{par2} using transversally elliptic operators.  Theorem
\ref{GS} is a consequence of these results.

In the present paper, we prove that $\dim Q(M,\CL^k)^G$ is
quasi-polynomial in $k$ for $k\geq1$ directly, and without making the
assumption that $0$ is a regular value of the moment map.  Our main
purpose is to show that this result may be obtained from the
Atiyah-Bott fixed point formula for $\chi_{\CL^k}$, using Theorem
\ref{first} as the only analytic input. The ideas underlying our paper
originated in the works of Paradan \cite{par2,parjump}.

\subsection{Contents of the paper}

The paper is structured as follows: in \S\ref{sec:fixed} we study the
calculus of expansions of the rational sum expression given for
$\chiek$ by the Atiyah-Bott fixed point formula. The main result is
Corollary \ref{abcharcor}, which gives the answer in terms of
partition functions.  We then proceed to introduce a quasi-polynomial
character $\Delta_\mu[\CE,\a]$, which encodes the asymptotic behavior
of this expansion. We begin \S\ref{sec:decomp} by Paradan's
combinatorial formula decomposing a partition function in terms of
convolution products of partitions functions in lower dimensions. Then
we apply this formula to our geometric setup (Proposition
\ref{paradan}), which results in a decomposition of $\chie$ in terms
of certain formal characters, which are enumerated by fixed-point sets
of subtori of $\T$. This combinatorial decomposition is the
$K$-theoretical analogue of the stratification of the manifold $M$ via
the Morse function $\|\mu\|^2$ used by Witten \cite{wit92} to compute
intersection numbers on reduced spaces.

We finish the proof of Theorem \ref{GS} in \S\ref{sec:quasi} by
studying the terms of this expansion. We quickly reduce the final
result to a numerical statement regarding the weights of the $T$-action at
fixed point sets of subtori. This statement is then proved via a
``localization of positivity'' result: Proposition
\ref{crucial}. Finally, we give a quick proof of Theorem \ref{first}
in \S\ref{sec:appendix}.  A list of notations given in \S\ref{lon}
helps the reader to navigate through the
paper.  

{\bf Acknowledgement.}  We are grateful to the referee for careful
reading of the article, and useful suggestions.

\section{Fixed point formula and a  formal character}
\label{sec:fixed}

As in the previous section, let us begin with a connected,
compact, almost complex $\T$-manifold $M$, and a pair $(\CE,\CL)$,
consisting of a complex equivariant vector bundle and a line
bundle on $M$. We assume again that the $T$-fixed points are isolated.

In this section, we embark on the study of the sequence
of characters $\chiek$, $k=0,1,\dots$.

\subsection{The fixed point formula}
\label{sec:fixed-point-formula}

Our starting point is the Atiyah-Bott fixed point formula
\cite{ati-bot2}, which expresses $\chiek$ as a sum of contributions
associated to the fixed points of the $T$-action on $M$.

Before we proceed, we need to introduce notation and terminology for
sets with multiplicities, which we will call  {\em lists}.  A list
$\Phi$ thus consists of a set $\{\Phi\}$, and a multiplicity function
$m_\Phi:\{\Phi\}\to\Z_{>0}$. We will use the notation
$[\phi_1,\phi_2,\dots,\phi_N]$ for the list of elements
$\phi_1,...$. We will also write
\begin{itemize}
\item $\psi\in\Phi$ if $\psi\in\{\Phi\}$; \item if $\psi\in\Phi$
and $m_\Phi(\psi)>1$, then $\Phi-\{\psi\}$ will
  denote the list $\Phi$ with the multiplicity of $\psi$ decreased by
  1; if $m_\Phi(\psi)=1$, then  $\Phi-\{\psi\}$ will
  denote the list $\Phi$ with $\psi$ removed;
\item for a list $\Phi$ and a set $S$, we will write $\Phi\cap S$
  for the list with underlying set $\{\Phi\}\cap S$ and
  multiplicity function coinciding with that of $\Phi$ on this set;
we will write $\Phi\setminus S$
  for the list with underlying set $\{\Phi\}\setminus S$ and
  multiplicity function coinciding with that of $\Phi$ on this set;
\end{itemize}

Now, denote by $F$ the finite set of fixed points of the $T$-action on
$M$.  For each fixed point $p\in F$, the weights of the $\T$-action on
the fiber $\CE_p$ form a list, which we will denote by $\Psi_p$.  Let
 $\ccp$ be the function $\T\to\C$ obtained by taking the trace of the
$\T$-action on the fiber $\CE_p$. Thus we have
 $\ccp=\sum_{\eta\in\Psi_p}e_{\eta}$.
Similarly, we denote by $\Phi_p$ the list of $\T$-weights of the
complex vector space $\tjs_pM$.

 With these preparations we can state the
Atiyah-Bott fixed point formula for our case:
\begin{equation}
  \label{ab}
\chie = \sum_{p\in F} \frac{\ccp}{\prod_{\phi\in
\Phi_p}(1-e_{\phi})}.
\end{equation}
This is an equality between two functions defined on an open and dense
subset of $\T$.  Indeed, the right hand side is meaningful on the set
\[\{t\in \T|\,t^\phi\neq1\;\forall p\in F\text{ and
}\phi\in\Phi_p\},\] while the left hand side is regular on $\T$.

Let us see two  examples.
First, we return to our Example \ref{ponex1}.
\begin{example}
  \label{baby0} Let $M=P^1(\C)$ with the action of $\mathrm{U}(1)$ given by
  $t\cdot(x:y)=(tx:t^{-1}y)$, and let $\CL^k$ be the $k$th tensor
  power of the dual of the tautological line bundle $\CL$. There are $2$
  fixed points $p^+=(1:0)$ and $p^-=(0:1)$, and we have
$$\chi_{\CL^k}(t)=\frac{t^k}{(1-t^{-2})}+
\frac{t^{-k}}{(1-t^2)}.$$

The graph of the function $\F\chi_{\CL^k}$ is pictured below for $k=4$.
\medskip

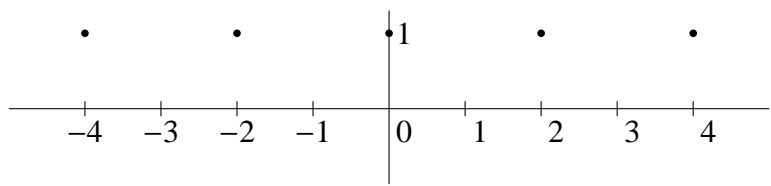
\begin{tikzpicture}

\draw (0,0) -- (10,0);
\draw (5,-1) -- (5,1.3);

\fill[color=black]  (1,1) circle (.5mm);
\fill[color=black]  (3,1) circle (.5mm);
\fill[color=black]  (5,1) circle (.5mm);
\fill[color=black]  (7,1) circle (.5mm);
\fill[color=black]  (9,1) circle (.5mm);

\draw (1,0.1) -- (1,-0.1);
\draw (2,0.1) -- (2,-0.1);
\draw (3,0.1) -- (3,-0.1);
\draw (4,0.1) -- (4,-0.1);
\draw (6,0.1) -- (6,-0.1);
\draw (7,0.1) -- (7,-0.1);
\draw (8,0.1) -- (8,-0.1);
\draw (9,0.1) -- (9,-0.1);

\draw (1,-0.3) node {$-4$};
\draw (2,-0.3) node {$-3$};
\draw (3,-0.3) node {$-2$};
\draw (4,-0.3) node {$-1$};
\draw (5.2,-0.3) node {$0$};
\draw (6.2,-0.3) node {$1$};
\draw (7.2,-0.3) node {$2$};
\draw (8.2,-0.3) node {$3$};
\draw (9.2,-0.3) node {$4$};
\draw (5.2,1) node {$1$};

\end{tikzpicture}
  

\end{example}

\begin{example}
\label{flag}

Let $M$ be the flag variety of $\C^3$ endowed with the action of the
group $\mathrm{U}(3)$.  The subgroup $\T_u=\{(t_1,t_2,t_3);\;
t_1,t_2,t_3\in \mathrm{U}(1)\}\subset \mathrm{U}(3)$ of diagonal
matrices is the maximal torus of $U(3)$, and the weight lattice of
$\T_u$ has a canonical diagonal decomposition:
$\Z\theta_1+\Z\theta_2+\Z\theta_3$.  The coordinate flag $$\{\C
e_1\subset \C e_1\oplus \C e_2 \subset \C e_1\oplus \C e_2\oplus \C
e_3\}$$ is fixed under $\T_u$, and the rest of the fixed points in
$M^{\T_u}$ may be obtained by applying to this flag the elements of
the permutations group $\Sigma_3$ in a natural manner. We will use the
notation $w\in\Sigma_3\mapsto p_w\in M^{\T_u}$ for this
correspondence; in particular, the coordinate flag will be denoted by
$p_{123}$.

Consider the line
bundle $\CL$ induced from the character $t_1^4t_2^{-1}t_3^{-3}$ of $\T_u$. Then
\[  \chi_{\CL^k}(t_1,t_2,t_3)=\sum_{w\in \Sigma_3}
w*\frac{t_1^{4k}t_2^{-k}t_3^{-3k}}{(1-t_2/t_1)(1-t_3/t_2)(1-t_3/t_1)},
\]
where, again, $w*$ stands for the natural action of $\Sigma_3$ on the indices.

In what follows, we consider $\chi_{\CL^k}$ as a character of the
adjoint group $G$ of ${\mathrm U}(3)$.  Let $T$ be the maximal torus
of $G$, with Lie algebra $\t$.  Then $\t^*$ has basis the simple roots
$\alpha=\theta_1-\theta_2$ and $\beta=\theta_2-\theta_3$ and the
weight lattice $\Lambda$ of $T$ is $\Z \alpha+\Z\beta.$ The weight
$\mu_{123}$ of the bundle $\CL$ at $p_{123}$ is $4\alpha+3\beta.$ The
multiplicity function $\F\chi_{\CL}$ on $\Lambda$ then may be
represented as follows.

\medskip
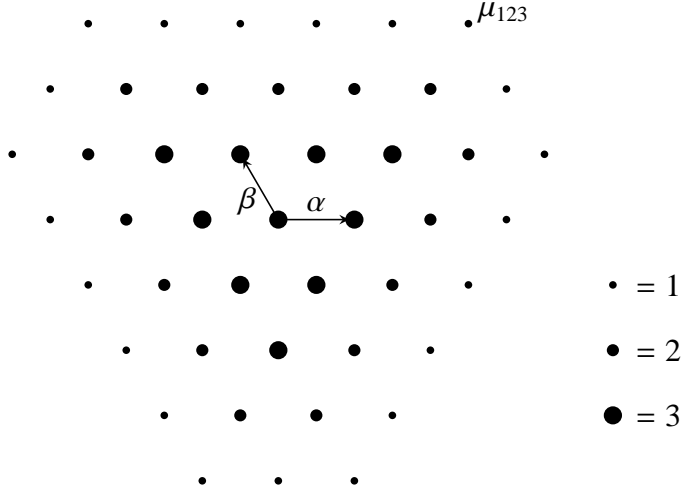
\begin{tikzpicture}

\fill[color=black] (4,1.732) circle (.5mm);
\fill[color=black] (5,1.732) circle (.5mm);
\fill[color=black] (6,1.732) circle (.5mm);

\fill[color=black] (3.5,2.6) circle (.5mm);
\fill[color=black] (4.5,2.6) circle (.8mm);
\fill[color=black] (5.5,2.6) circle (.8mm);
\fill[color=black] (6.5,2.6) circle (.5mm);

\fill[color=black] (9.4,2.6) circle (1.2mm);
\draw (10,2.6) node {$=3$};

\fill[color=black] (3,3.464) circle (.5mm);
\fill[color=black] (4,3.464) circle (.8mm);
\fill[color=black] (5,3.464) circle (1.2mm);
\fill[color=black] (6,3.464) circle (.8mm);
\fill[color=black] (7,3.464) circle (.5mm);

\fill[color=black] (9.4,3.464) circle (.8mm);
\draw (10,3.464) node {$=2$};

\fill[color=black] (2.5,4.33) circle (.5mm);
\fill[color=black] (3.5,4.33) circle (.8mm);
\fill[color=black] (4.5,4.33) circle (1.2mm);
\fill[color=black] (5.5,4.33) circle (1.2mm);
\fill[color=black] (6.5,4.33) circle (.8mm);
\fill[color=black] (7.5,4.33) circle (.5mm);

\fill[color=black] (9.4,4.33) circle (.5mm);
\draw (10,4.33) node {$=1$};

\fill[color=black] (2,5.196) circle (.5mm);
\fill[color=black] (3,5.196) circle (.8mm);
\fill[color=black] (4,5.196) circle (1.2mm);
\fill[color=black] (5,5.196) circle (1.2mm);
\fill[color=black] (6,5.196) circle (1.2mm);
\fill[color=black] (7,5.196) circle (.8mm);
\fill[color=black] (8,5.196) circle (.5mm);

\fill[color=black] (1.5,6.062) circle (.5mm);
\fill[color=black] (2.5,6.062) circle (.8mm);
\fill[color=black] (3.5,6.062) circle (1.2mm);
\fill[color=black] (4.5,6.062) circle (1.2mm);
\fill[color=black] (5.5,6.062) circle (1.2mm);
\fill[color=black] (6.5,6.062) circle (1.2mm);
\fill[color=black] (7.5,6.062) circle (.8mm);
\fill[color=black] (8.5,6.062) circle (.5mm);

\fill[color=black] (2,6.928) circle (.5mm);
\fill[color=black] (3,6.928) circle (.8mm);
\fill[color=black] (4,6.928) circle (.8mm);
\fill[color=black] (5,6.928) circle (.8mm);
\fill[color=black] (6,6.928) circle (.8mm);
\fill[color=black] (7,6.928) circle (.8mm);
\fill[color=black] (8,6.928) circle (.5mm);

\fill[color=black] (2.5,7.794) circle (.5mm);
\fill[color=black] (3.5,7.794) circle (.5mm);
\fill[color=black] (4.5,7.794) circle (.5mm);
\fill[color=black] (5.5,7.794) circle (.5mm);
\fill[color=black] (6.5,7.794) circle (.5mm);
\fill[color=black] (7.5,7.794) circle (.5mm);

\draw  (7.97,7.94)  node {$\mu_{123}$};

\draw (5.5,5.4) node {$\alpha$};
\draw (4.6,5.45) node {$\beta$};

\begin{scope}[->,>=stealth,semithick]
\draw[->]  (5,5.196)   -- +(0:.93cm);
\draw[->]  (5,5.196)   -- +(120:.93cm);
\end{scope}

\end{tikzpicture}
\medskip

\end{example}

\subsection{The partition function}
\label{sec:partition-function}

Recall that $\chie$ is determined by its Fourier coefficients
$\F\chie:\Lambda\to\Z$, and that this latter function has finite
support in $\Lambda$.  Our immediate goal is to convert the equality
\eqref{ab} into an equality of two functions in the Fourier dual space
of $\Z$-valued functions on $\Lambda$. For this task, we follow the
same method as \cite{gui-ler-ste}, \cite{gui-pra}.

 In this paragraph, we make
the additional assumption that the generic stabilizer of the
$\T$-action on $M$ is finite; this is equivalent to the condition that
$\Phi_p$ spans $\t^*$ for all $p\in F$.

Before we proceed, we need to introduce a few basic notions.
\begin{itemize}
\item We denote by $R(\T)$ the set of finite integral linear
  combinations of the characters $e_{\lambda},
  \,\lambda\in\Lambda$,  and
\item by $\hat R(\T)$ the space of formal, possibly infinite, integral
  linear combinations of these characters. Thus the elements of $\hat
  R(\T)$ are in one-to-one correspondence with the functions
  $m(\lambda):\Lambda\to\Z$ via $\theta:=\sum_{\lambda\in \Lambda}
  m(\lambda) e_{\lambda}\in\hat R(\T)$. We will write $\F\theta$ for
  the function $m$ in this case. Conversely, given a function $m$, we
  will call the corresponding series $\theta$ its character. If we
  extend the weights $\lambda\in\Lambda$ to linear functions on
  $\t_\C$, then we can also think of the elements of $\hat R(\T)$ as
  formal series of holomorphic exponential functions on $\t_{\C}$.
\item  Informally, we will call $\delta\in\rhat$ a {\em quasi-polynomial
    character} if its Fourier transform $\F\delta:\Lambda\to\Z$ is
  quasi-polynomial (cf. Definition \ref{defquasi}).
\end{itemize}

We collect some simple observations needed later.

\begin{lemma} \label{rprops}
  \begin{enumerate}
  \item  $\hat R(\T)$ is a module over $R(\T)$, and the set of
quasi-polynomial characters forms a linear subspace in $\rhat$ which is
stable under multiplication by $R(\T)$.
\item Elements of $\hat R(\T)$ whose Fourier transforms are supported
  on a fixed acute cone in $\Lambda$ may be multiplied, thus they form
  a ring.
\item For $\Theta\in \hat R(T)$ and $\lambda,\mu \in \Lambda$, we have
  $\F(e_\mu \Theta)(\lambda)=\F\Theta(\lambda-\mu)$.
\item If a quasi-polynomial function $f$ on $\Lambda$ vanishes at all
  points of a set $Q\cap \Lambda$, where $Q$ is a non-empty  open cone, then
  $f=0$.
  \end{enumerate}
\end{lemma}

The proofs are straightforward and will be omitted.  With these
preparations, we are ready to introduce the basic building block of
our constructions.  For a list of weights $\Phi$, we will need to
represent the function $\prod_{\phi\in \Phi}(1-e_\phi)^{-1}$ by an
element of $\hat R(T)$.  To this end, we can expand each factor of the
form $(1-e_\phi)^{-1}$ as a geometric series, but this product is only
meaningful in the ring $\hat R(T)$ if $\Phi$ lies in an acute cone. To
remedy this problem, we will reverse the signs of some of the vectors
in $\Phi$, which, in turn, necessitates the introduction of the notion
of polarization.

  Let $\Phi$ be a list of nonzero elements of $\Lambda$.  We will call
  $\Xpol\in \t$ {\em polarizing} for $\Phi$ if $\langle
  \phi,\Xpol\rangle \neq 0$ for every $\phi\in \Phi$. For nonempty
  $\Phi$ and polarizing $\Xpol$, split $\Phi$ in $\Phi=\Phi_+\cup
  \Phi_-$, where
  \[ \Phi_+=\{\phi\in \Phi \, | \, \langle \phi,\Xpol\rangle >0\}\quad
  \text{and} \quad \Phi_-= \{\phi\in \Phi\,|\, \langle \phi,\Xpol\rangle
  <0\},\] and introduce the formal character
 \begin{equation}\label{deftheta}
   \Theta[\Phi\felarrow \Xpol]=(-1)^{|\Phi_-|}
   \prod_{\phi\in\Phi_-}e_{-\phi}\times \prod_{\phi\in
     \Phi_-}\left(\sum_{k=0}^\infty e_{-k\phi}\right)
   \times\prod_{\phi\in
     \Phi_+}\left(\sum_{k=0}^\infty e_{k\phi}\right) .\end{equation}
 It is easy to verify that the products in this formula are
 meaningful, and hence
 the series $\Theta[\Phi\felarrow \Xpol]$ defines an element of $\hat
 R(\T)$. We also set  \mbox{$\Theta[\emptyset\felarrow \Xpol]=1$} for
 any $\Xpol\in \t$.

 The notation $\Theta[\Phi\felarrow \Xpol]$ represents the fact that
 we have reoriented the elements of $\Phi$ using $\Xpol$. Note,
 however, that $\Theta[\Phi\felarrow \Xpol]$ coincides with
 $\Theta[\Phi_Y]$, where $\Phi_Y$ is the reoriented list, up to a sign
 and a shift only. These are motivated by the following

\begin{lemma} \label{thetaconv}
  \begin{enumerate}
  \item $\F\Theta[\Phi\felarrow \Xpol]$ is supported on the pointed cone generated
    in $\t^*$ by the set $\Phi_{+}\cup (-\Phi_{-})$, in particular,
    apart from the origin, on the half-space $\{Y>0\}$.
  \item As a  formal character, $\Theta[\Phi \felarrow \Xpol]\in\rhat$
    satisfies \[ \Theta[\Phi\felarrow \Xpol]\cdot\prod_{\phi\in
      \Phi}(1- e_{\phi})\, =1.\]
  \item Considered as a series
    of holomorphic functions on the complexification $T_\C$ of the
    torus group $T$,  the series \eqref{deftheta}
    converges absolutely, in a neighborhood of the point
    $\exp(i\Xpol)\in T_\C$, to the function $\prod_{\phi\in
      \Phi}(1-e_\phi )^{-1}.$
  \end{enumerate}

\end{lemma}
The proofs are straightforward and are left to the reader. Using these
facts, we can rewrite \eqref{ab} as follows.

\begin{corollary}[\cite{gui-ler-ste,gui-pra}] \label{abcharcor} For a
  vector $\Xpol$, which is polarizing for the union $\cup_{p\in
    F}\Phi_p$ of the lists $\Phi_p$, the following equality holds in
  $\hat R(\T)$:
 \begin{equation}
  \label{abchar}
\chie=\sum_{p\in F} \ccp \cdot\Theta[\Phi_p\felarrow \Xpol].
\end{equation}
 \end{corollary}
\noindent Indeed, multiplying the right hand side of \eqref{ab} and
 \eqref{abchar} by 
\[\prod_{p\in F}\prod_{\phi\in\Phi_p}(1-e_\phi),\] we
 obtain the same result. On the other hand, it is easy to see that the
 operation of multiplication by this product is injective on the subspace of
 elements of $\rhat$ which are supported on a half-space bounded by a
 hyperplane orthogonal to $Y$.




\begin{remark}
  The function $\F\Theta[\Phi\felarrow \Xpol]:\Lambda\to\Z$,
  traditionally, has been called the {\em partition function}, since,
  assuming $\Phi=\Phi^+$, its value at $\mu$ equals the number of ways
  one can write $\mu$ as a nonnegative integral linear combinations of
  vectors from $\Phi$. In particular, the equality \eqref{abchar}
  applied to Weyl's formula for the characters leads to Kostant's
  formula for the multiplicity of a weight in an irreducible
  representation of a reductive Lie group.
\end{remark}

A key fact is that the Fourier transform $\F\Theta[\Phi\felarrow
\Xpol]$, as a function on $\Lambda$, is piecewise
quasi-polynomial. Let us describe this in more detail:
\begin{definition}\label{defphireg}
Given a list $\Phi$ spanning $\t^*$, we will call an element
$\gamma\in\t^*$ $\Phi$-{\em regular} if it is not the linear
combination of fewer than $\dim(\t)$ elements of $\Phi$.
\end{definition}

The set of $\Phi$-regular elements form the complement of a hyperplane
arrangement in $\t^*$, and we will use the term $\Phi$-{\em tope} for
the connected components of this set\footnote{We use the word tope, as our definition is
similar to the notion of tope in matroid theory.}. It will be convenient to use the
notation $\vv(\gamma)$ for the tope containing the $\Phi$-regular
element $\gamma$. Note that topes are open convex cones, which are
invariant under rescaling.

\begin{lemma}[\cite{DM}, see also \cite{CPV}] \label{qpoly} Let $\Phi$
  be a list of nonzero vectors spanning $\t^*$, let $\Xpol$ be a
  polarizing vector for $\Phi$, choose a $\Phi$-tope $\vv$. Then there
  exists a unique quasi-polynomial character $\qpol[\Phi\felarrow
  \Xpol,\vv]$, whose Fourier transform $\F\qpol[\Phi\felarrow
  \Xpol,\vv]$ coincides with $\F\Theta[\Phi\felarrow \Xpol]$ on
  $\Lambda\cap\vv$.
\end{lemma}

\begin{remark}
  \label{topedegen}
  This lemma may be naturally extended to the situation when $\Phi$
  does not span $\t^*$. In this case, denoting the smallest linear
  subspace of $\t^*$ containing $\Phi$ by $\lspan(\Phi)$, the tope
  $\vv$ is in $\lspan(\Phi)$, and $\qpol[\Phi\felarrow
  \Xpol,\vv]$ is a function supported on $\lspan(\Phi)$, whose
  restriction to $\lspan(\Phi)$ is quasipolynomial. The degree of the
  quasi-polynomial $\F\qpol[\Phi\felarrow \Xpol,\vv]$ is equal to
  $|\Phi|-\dim \lspan(\Phi)$.
\end{remark}

\begin{example}\label{example:quasi}
  Let $\t^*=\R \alpha,\, \Lambda=\Z \alpha$, $\Phi=[\alpha]$ and let
  $\Xpol\in \t$ to be the vector satisfying $\langle
  \alpha,\Xpol\rangle =1$. Then $$\Theta[\Phi\felarrow
  \Xpol]=\sum_{k=0}^{\infty}e_{k\alpha}.$$ Then
  $\vv^+:=\{t\alpha,t>0\}$, $\vv^-:=\{t\alpha,t<0\}$ are topes. The
  function $\F\Theta[\Phi\felarrow \Xpol]$ coincides with the constant
  function $1$ on $\Z \alpha\cap \vv^+$ and with $0$ on $\Z \alpha\cap
  \vv^-$.  The character $\delta=\sum_{k\in \Z} e_{k\alpha}$ is
  quasi-polynomial as the multiplicity $\F\delta$ is the constant
  function $1$ on $\Z \alpha$.  Thus
\[\qpol[\Phi\felarrow
  \Xpol,\vv^+]=\sum_{k\in \Z} e_{k\alpha},\; \text{while }\qpol[\Phi\felarrow
  \Xpol,\vv^-]=0.\]
\end{example}

\subsection{The asymptotics of the character}
\label{sec:asympt-char}

We return to our geometric setup. We continue to assume that the torus
$\T$ acts on the compact almost complex manifold $M$ with a finite set
of fixed points. We consider a Hermitian $\T$-equivariant line bundle
$\CL$, a complex equivariant vector bundle $\CE$, and we study the character
$\chiek$.

We choose an equivariant Hermitian connection on $\CL$. Recall from
\S\ref{sec:thm} that $\mu(p)$, the value of the associated moment map
$\mu:M\to\t^*$ at a fixed point $p\in F$,  is the weight of the
$\T$-action on the fiber $\CL_p$.  Thus, in this instance, formula
\eqref{abchar} may be written in the form
 \begin{equation}
  \label{abcharlk}
\chi_{\CL}=\sum_{p\in F} e_{\mu(p)} \Theta[\Phi_p\felarrow \Xpol],
\end{equation}
 and hence
\begin{equation}
  \label{abcharlkfourier}
\F\chi_{\CL}(\lambda)=\sum_{p\in F} \F\Theta[\Phi_p\felarrow \Xpol](\lambda-\mu(p)).
\end{equation}

Now assume that the generic stabilizer of the action of $T$ on $M$ is
finite, or, equivalently, that $\Phi_p$ spans $\t^*$ for all $p\in
F$.  Then the moment map $\mu$ gives rise to a real affine hyperplane
arrangement whose complement is the open set
\begin{equation}
  \label{eqalcove}
\bigcap_{p\in
  F}\left\{\gamma\in\t^*|\,\gammap\text{ is }
\Phi_p\text{-regular}\right\}\subset\t^*.
\end{equation}  We will use the
  term {\em alcove} for the connected components of the set
  \eqref{eqalcove}.
 The alcoves are thus
  minimal nonempty intersections of the translated polyhedral cones
  $\vv+\mu(p)$, where $p\in F$, and $\vv$ is a tope of $\Phi_p$.  Just
  as in the case of topes, we will use the notation $\a(C)$ for the
  alcove containing the connected subset $C$ of the set \eqref{eqalcove}.

   \begin{remark}[\cite{atiyah,gs}]
     If $\CL$ is a positive line bundle (cf. Definition
     \ref{defpositive}), then $\mu(M)$ is the convex hull of the set
     of points $\{\mu(p);\;p\in F\}$, and the set (\ref{eqalcove}) is
     contained in the set of regular values of $\mu$.
\end{remark}

Next, we define a quasi-polynomial character
 $\loc_\mu[\CE,\a]$ by formally replacing the
generating function for the partition function $\Theta[\Phi_p\felarrow
\Xpol]$ in \eqref{abchar} by an appropriately chosen
quasi-polynomial $\delta[\Phi_p\felarrow \Xpol,\vv]$ (cf. Lemma \ref{qpoly}).

\begin{definition}\label{defDelta}
  Given a $\T$-equivariant vector bundle $\CE$ over $M$, and an alcove
  $\a\subset\t^*$, we define the formal character
  \begin{equation}
    \label{defD}
\loc_\mu[\CE,\a] = \sum_{p\in F}
  \ccp\cdot\delta[\Phi_p\felarrow \Xpol,\vv(\ap)],
      \end{equation}
where $\ccp$, as usual, stands for the sum of $T$-weights of the fiber $\CE_p$. 
\end{definition}
\begin{remark}\mbox{}
  Note that we omitted the dependence on $\Xpol$ in the notation (cf.
 Corollary \ref{indep}).
\end{remark}

The meaning of this object will become clear after Proposition \ref{Dchi}.
Note that since  $\loc_\mu[\CE,\a]$ is a linear combination of
quasi-polynomial characters, it is itself quasi-polynomial.

\begin{lemma}\label{restrict}
The quasi-polynomial
$\F\loc_\mu[\CL,\a]$ coincides with $\F\chi_{\CL}$ at all points of $\a\cap \Lambda$.
\end{lemma}
\begin{proof}
Indeed, since $\tau_p(\CL)=\mu(p)$, we have
$$ \F\loc_\mu[\CL,\a](\lambda) = \sum_{p\in F}
\F\delta[\Phi_p\felarrow \Xpol,\vv(\ap)](\lambda-\mu(p)).$$ On the
other hand, by the definition of $\delta[\Phi\felarrow \Xpol,\vv]$, if
$\lambda$ belongs to the alcove $\a$, then
\[
\F\delta[\Phi_p\felarrow \Xpol,\vv(\ap)](\lambda-\mu(p))=
\F\Theta[\Phi_p\felarrow \Xpol](\lambda-\mu(p)).
\]
Now \eqref{abcharlkfourier} immediately implies the statement of the
Lemma.
\end{proof}

\begin{remark}
The finite set $\a\cap \Lambda$ may be small, even empty, hence we
cannot necessarily determine $\loc_\mu[\CL,\a]$ by restricting the
quasi-polynomial function $\F\loc_\mu[\CL,\a]$ to this set.
\end{remark}

 \begin{example}
\label{flagalcoves}
We return to Example \ref{flag}, with $\mu$ associated to the line
bundle $\CL$.  The diagram depicts the dual of the Lie algebra of the
maximal torus of the adjoint group of $U(3)$.  The straight lines cut
the plane into alcoves. The support of the multiplicity function
$\F\chi_{\CL}$ is the highlighted hexagon, and the function is
invariant under the symmetries of this hexagon.
\end{example}

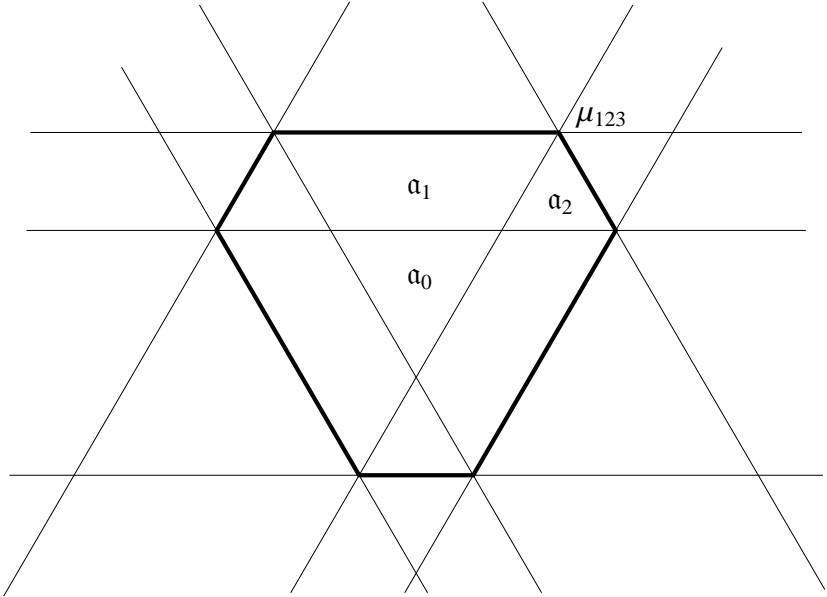
\begin{tikzpicture}
\coordinate (a) at (3,1.3);
\coordinate (b) at (4.5,1.3);
\coordinate (c) at (6.375,4.5465);
\coordinate (d) at (5.625,5.8455);
\coordinate (e) at (1.875,5.8455);
\coordinate (f) at (1.125,4.5465);

\begin{scope}[ultra thin]
\draw (a)  -- +(-60:1.8cm);
\draw (a)  -- +(180:4.6cm);
\draw (a)  -- +(-120:1.8cm);
\draw (a)  -- +(60:7.2cm);

\draw (b)  -- +(-120:1.8cm);
\draw (b)  -- +(0:4.6cm);
\draw (b)  -- +(-60:1.8cm);
\draw (b)  -- +(120:7.2cm);

\draw (c)  -- +(60:2.8cm);
\draw (c)  -- +(-60:5.5cm);
\draw (c)  -- +(0:2.5cm);
\draw (c)  -- +(180:7.75cm);

\draw (d)  -- +(120:2cm);
\draw (d)  -- +(0:3.2cm);

\draw (e)  -- +(60:2cm);
\draw (e)  -- +(180:3.2cm);

\draw (f)  -- +(120:2.5cm);
\draw (f)  -- +(-120:5.6cm);
-- (2,1);
\end{scope}

\draw[ultra thick] (a) -- (b) -- (c) -- (d) -- (e) -- (f)
-- cycle;

\draw (3.8,3.9)   node {$\a_0$};
\draw (3.8,5.1)   node {$\a_1$};
\draw (5.65,4.9)   node {$\a_2$};
\draw  (6.2,6.05)   node {$\mu_{123}$};
\end{tikzpicture}

For this example,
the quasi-polynomials are polynomials, and can be guessed by ``interpolation''   from the picture of $\F\chi_{\CL}$ given in  Example \ref{flag}.
We have
 \begin{eqnarray*}
 \F\loc_\mu[\CL,\a_0](n_1\alpha+n_2\beta)&=&3,\\
 \F\loc_\mu[\CL,\a_1](n_1\alpha+n_2\beta)&=&4-n_2,\\
 \F\loc_\mu[\CL,\a_2](n_1\alpha+n_2\beta)&=&5-n_1.\\
  \end{eqnarray*}
\qed

We now replace $\CE$ by $\CEk$.
\begin{lemma}  \label{delqpol}
  The function $(\lambda,k)\mapsto \F\loc_\mu[\CEk, \a](\lambda)$ is
  quasi-polynomial on the lattice $\Lambda\times \Z$.
\end{lemma}
\begin{proof}

Recall that, for $p\in F$, we denoted by  $\Psi_p$ the list of
$T$-weights of the fiber $\CE_p$, and we set $\ccp=\sum_{\eta\in\Psi_p}e_\eta$.
Clearly,  we have $\tau_p[\CEk]=\ccp\cdot  e_{k\mu(p)}$.
For a formal character
$\theta\in\hat R(\T)$ and $\lambda,\mu\in\Lambda$, the identity
$\F e_{k\mu}\theta(\lambda)=\F\theta(\lambda-k\mu)$ holds. This implies that 
$$ \F\loc_\mu[\CEk, \a](\lambda)=
\sum_{\eta\in\Psi_p}\F\delta[\Phi_p\felarrow \Xpol,\ap](\lambda-\eta-k\mu(p)).$$

As $\delta$ is a quasi-polynomial character, each term on the right hand side is a
quasi-polynomial function of $(\lambda,k)$, and this completes the proof.
\end{proof}

For small $k$, in particular for $k=0$, $\loc_\mu[\CEk,\a]$ does
not have any direct relationship with $\chiek$.
We have, nevertheless, the following asymptotic analog of
Lemma \ref{restrict}.

\begin{proposition} \label{Dchi} Let $\b$ be a compact subset of an
  alcove $\a$. Then there exists a positive integer $K$ such that for
  every $k>K$ and $\lambda\in k\b\cap \Lambda$, the equality
 \begin{equation}
   \label{Deltachi}
\F\loc_\mu[\CEk,\a](\lambda) = \F\chiek(\lambda)
 \end{equation}
holds.
\end{proposition}
\begin{proof}
Recall that $\Psi_p$ is the list of $T$-weights of the fiber $\CE_p$,
and $\ccp=\sum_{\eta\in\Psi_p}e_\eta$. 
According to \eqref{abchar}, we have
\[ \F\chiek(\lambda) = \sum_{p\in F}
\sum_{\eta\in\Psi_p}\F\Theta[\Phi_p\felarrow \Xpol](\lambda-\eta-k\mu(p)),
\]
while, by Lemma \ref{delqpol},
\[  \F\loc_\mu[\CEk,\a](\lambda) = \sum_{p\in F}
\sum_{\eta\in\Psi_p}\F\delta[\Phi_p\felarrow
\Xpol,\vv(\a-\mu(p))](\lambda-\eta-k\mu(p)).
\]
Hence, by the definition of the quasi-polynomial character
$\delta$ given in Lemma \ref{qpoly}, these two expressions
coincide as long as for each $p\in F$ and $\eta\in\Psi_p$, we have
$\lambda-\eta-k\mu(p)\in\vv(\a-\mu(p))$. Since topes are invariant under
rescaling,  we can conclude
that \eqref{Deltachi} holds if
\begin{equation}
  \label{rewr}
  \frac\lambda k-\frac\eta k\in\a\quad\text{for each }
\eta\in\cup_{p\in F}\Psi_p.
  \end{equation}

  As the set $\cup_{p\in F}\Psi_p$ is finite, for large enough $k$, we
  will have $\b-\eta/k\subset\a$ for every $\eta$ from this set. Hence
  \eqref{rewr} holds for large enough $k$, uniformly in
  $\lambda\in k\b\cap \Lambda$. This completes the proof.
\end{proof}
\begin{corollary} \label{indep} The quasi-polynomial character
  $\loc_\mu[\CE,\a]$ (cf. Definition \ref{defDelta}) does not depend
  on the choice of the polarizing vector $\Xpol$.
\end{corollary}
Indeed, note that Proposition \ref{Dchi} holds independently of the
vector $Y$ chosen to define $\loc_\mu[\CEk,\a]$, and, according to
Lemma \ref{delqpol}, $\F\loc_\mu[\CEk,\a]$ is quasi-polynomial on
$\Lambda\times \Z$. Now, by choosing an appropriate $\b$ with nonempty
interior in Proposition \ref{Dchi}, one can conclude that this
quasipolynomial restricted to $\{k;\;k>K\}\b\cap \Lambda$ is the same
for all choices of polarizing vectors $Y$. Now, the statement follows, since the
restriction to such an open set determines a quasipolynomial
(cf. Lemma \ref{rprops} (4)).
\medskip

Let us summarize what Proposition \ref{Dchi} says about
$\F\loc_\mu[\CE,\a]$. Consider the function $(k,\lambda)\mapsto
\F\chiek(\lambda)$, and interpolate its values on $\Z\times\Lambda$
from the values on the sets $k\b\cap \Lambda$ for $k$ sufficiently
large. This will result in a quasi-polynomial function, which is defined for all
$(k,\lambda)$. Then, the restriction to $k=0$ of this quasi-polynomial
function gives us our function $\F\loc_\mu[\CE,\a]$.

\bigskip

\begin{remark}
  One can give the following geometric interpretation to the character
  $\loc_\mu[\CE,\a]$ when the moment map $\mu: M\to \t^*$
  is associated to a positive line bundle.
  In this case, the curvature form $\Omega$ is non-degenerate, and $\mu$ is the moment
  map for the corresponding Hamiltonian structure on $M$.  Then any
  element $\gamma$ in an alcove $\mathfrak a$ is a regular value of
  $\mu$, and the torus $\T$ acts with finite stabilizers on
  $\mu^{-1}(\gamma)$.  The quotient $\mu^{-1}(\gamma)/\T$ is the same
  orbifold for all $\gamma\in\mathfrak a$, and thus
  we can denote it by $M_\a$.

  The bundle $\CE$ descends to an orbifold bundle $\CE_\a$ on $M_\a$,
  and each character $\lambda$ allows us to twist $\CE_\a$ by the
  associated line bundle $L_\lambda=\mu^{-1}(\gamma)\times_\T
  \C_{\lambda}$ over $M_{\a}$.  According to the index formula for
  orbifolds (\cite{atiyah-ell}, see also \cite{ver02}), the function
  $\lambda\rightarrow\dim Q(M_\a,\CE_\a\otimes L_\lambda)$ is
  quasi-polynomial. It can be easily shown using the results of
  \cite{mei1} that, in this setup, the character $\loc_\mu[\CE,\a]$
  appears as the generating function of this quasi-polynomial:
\[ \loc_\mu[\CE,\a]  =
\sum_{\lambda} \dim Q(M_\a,\CE_\a\otimes L_\lambda)\, e_\lambda.
\]
We will not use this geometrical interpretation in the present article.\qed
\end{remark}

In what follows, we will need the extension of the definition of
$\loc_\mu[\CE,\a]$ to the case when the generic stabilizer of the
$\T$ action on the connected manifold $M$ is not finite.

\begin{definition}
  Let the Lie group $G$ with Lie algebra $\g$ act on a
  manifold $M$.  Then, for a subset $C\subset M$, we denote by
\[    \g_C=\{X\in\g;\; VX\text{ vanishes on }C\}
\]
 and by $G_C$ the connected subgroup
of $G$ with Lie algebra $\g_C$.
\end{definition}

In our set up then, $\T_M$ is the {\em connected
  component} of the generic stabilizer of $M$ containing the identity
element, $\t_M\subset\t$ is the Lie algebra of $\T_M$, and for
every $p\in F$, the weights $\Phi_p$ span the annihilator
$\rzm\subset\t^*$.

Clearly, the group $\T_M$ acts on each of the fibers $\CE_q$, $q\in M$,
and since $M$ is connected, this representation does not depend on
$q$. In particular, for two fixed points $p,q\in F$, the
weights $\mu(p)$ and $\mu(q)$ of $\T$ differ by an element of
$\rzm$, and hence the affine-linear subspace
\begin{equation}
   \label{muaff}
   \am = \mu(p)+\rzm
 \end{equation}
 of $\t^*$ does not depend on $p\in F$. Note that, according to
 equation \eqref{musat}, the image $\mu(M)$ is contained in
 $\am$.

 Now we can repeat the definitions given in \eqref{eqalcove} and \eqref{defD}
 with $\t^*$ replaced by $\rzm$. More precisely, we consider the open
 set in $\am$ consisting of those elements $\gamma$ for which
 $\gamma-\mu(p)$ is $\Phi_p$-regular for any $p\in F$. An alcove
 $\a\subset \am$ is a connected component of this open set.
  As before, for an alcove $\a$,
 we denote by $\vv(\a-\mu(p))$ the $\Phi_p$-tope in $\rzm$ containing
 $\a-\mu(p)$. The formal character $\loc_\mu[\CE,\a]$ may be defined by
 equation \eqref{defD} (here we choose any polarizing vector in $\t$):
 \begin{equation}
\loc_\mu[\CE,\a] = \sum_{p\in F}
  \ccp \cdot\delta[\Phi_p\felarrow \Xpol,\vv(\ap)].
      \end{equation}

      Note that the function $\F\delta[\Phi_p\felarrow
      \Xpol,\vv(\ap)]$ is supported on $\rzm\cap \Lambda$, while the
      weights in $\Psi_p$ do not necessarily belong to $\rzm\cap
      \Lambda$. Thus the multiplicity function $\F\loc_\mu[\CE,\a]$ is
      supported on a finite number of translates of $\rzm\cap
      \Lambda$, and it is quasi-polynomial on each translate.

 Denote by $\C_\lambda$ the trivial line bundle over $M$ endowed with
 the action $e_{\lambda}$ of $\T$.  For any equivariant bundle $\CE$ over
 $M$, we have a decomposition
\begin{equation}
\CE=\bigoplus_{\lambda\in \Lambda/\Lambda\cap \rzm} \C_{\lambda}\otimes
(\CE\otimes \C_{-\lambda})^{\T_M},
\end{equation}
where the sum is understood as taken over any system of
representatives of the quotient. This leads to the formula
\begin{equation}\label{EE}
\loc_\mu[\CE,\a]=\sum_{\lambda\in \Lambda/\Lambda\cap \rzm} e_{\lambda}
\loc_\mu[(\CE\otimes \C_{-\lambda})^{\T_M},\a],
\end{equation}
which expresses the formal $\T$-character $\loc_\mu[\CE,\a]$ through
quasi-polynomial characters of the torus $\T/\T_M$. Formula \eqref{EE}
has the following simple corollary:
\begin{lemma} \label{simple}
  If for some $\lambda \in \Lambda$, the multiplicity
  $\F\loc_\mu[\CE,\a](\lambda)$ is not zero, then the
  restriction of $\lambda$ to $\t_M$ is a weight of the representation
  of $\T_M$ on a fiber of $\CE$.
\end{lemma}

We end this section with a quick study of the situation when  the
affine space $\am$ given by equation \eqref{muaff} is linear, i.e.
passes through the origin.  This is equivalent to the condition
that $\T_M$ acts trivially on the fibers of $\CL$, i.e. $\CL$ is a
$\T/\T_M$-line bundle.

\begin{lemma}\label{quasipolynomial}
Let $\CE$ be a $\T$-bundle, and  $\CL$ be a  $\T/\T_M$-line bundle on $M$.
Then $k\mapsto \F\loc_\mu[\CE\otimes \CL^k,\a](0)$
is a quasi-polynomial function of $k$.
\end{lemma}
\begin{proof}
Applying  (\ref{EE}) to the bundle $\CE\tensor\CL^k$,  and using the condition on $\CL$, we obtain the equality
\[ \loc_\mu[\CE\otimes \CL^k,\a](0) = \loc_\mu[\CE^{\T_M}\otimes
\CL^k,\a](0).
\]
Since $\CE^{\T_M}$ is a $\T/\T_M$-equivariant vector bundle,  we can
replace $\T$ by $\T/\T_M$. According to Lemma \ref{delqpol},
$\F\loc_\mu[\CE^{\T_M}\otimes \CL^k,\a](\lambda)$ is quasi-polynomial in
$(\lambda,k)\in (\rzm\cap \Lambda)\times\Z$, and hence
$\F\loc_\mu[\CE^{\T_M}\otimes \CL^k,\a](0)$ is a quasi-polynomial function of
$k$.
  \end{proof}

\section{Decomposition of partition functions}
\label{sec:decomp}

In this section, we prove a decomposition formula for the generating
function $\Theta[\Phi\felarrow \Xpol]$ of the partition function
introduced in \eqref{deftheta}. This formula is due to Paradan and it
will serve as the combinatorial engine of our proof of Theorem
\ref{GS}.

\begin{definition}
  Given a list $\Phi$ of weights in $\Lambda\subset\t^*$, introduce
  the set of {\bf $\Phi$-rational subspaces}
\[     \CRF = \{S\subset\t^*\text{ linear};\;  \Phi\cap S \text{ spans }S\}.
\]
This is the set of linear subspaces of $\t^*$ spanned by some subset
of $\Phi$:
\end{definition}

\begin{remark} \label{remrs}
\noindent 1. Note that $\{0\}\in\CRF$, and $\t^*\in\CRF$ if $\Phi$
spans $\t^*$.   \\
\noindent 2. Comparing this definition to Definition
\ref{defphireg}, we see that all subspaces $S\in\CRF$, except for
$S=\t^*$, consist of non-regular elements.
\end{remark}

Fix a positive definite scalar product $(\cdot,\cdot)$ on
$\t^*$. This will allow us to define orthogonal projections in $\t^*$,
as well as to identify $\t$ and $\t^*$ whenever necessary.

For each rational subspace $\r\in \CRF$ and
$\gamma\in\t^*$, introduce the notation $\gamma_\r$ for the orthogonal
projection of $\gamma$ onto $\r$, and $\ysg$ for the vector
$(\gamma_\r-\gamma)$ (see the diagram below).
Thus we have the orthogonal decomposition $$\gamma=\gamma_S-Y_{S,\gamma}.$$

In what follows, we will consider $Y_{S,\gamma}$ to be an element of
$\t$. 

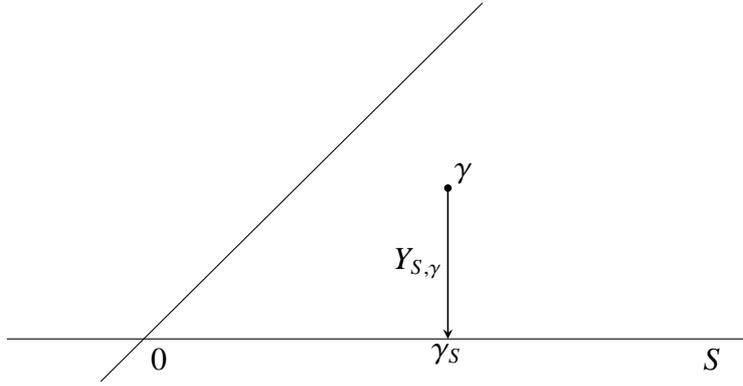
\begin{tikzpicture}
\coordinate (a) at (1,0);
\coordinate (b) at (5,2);

\draw (a)  -- +(0:8cm);
\draw (a)  -- +(180:1.8cm);
\draw (a)  -- +(-135:0.8cm);
\draw (a)  -- +(45:6.3cm);

\draw (1.2,-.27) node {$0$};
\draw (5.2,2.2) node {$\gamma$};
\draw (8.5,-.27) node {$S$};
\draw (5,-.2) node {$\gamma_S$};
\draw (4.6,1) node {$\ysg$};

\begin{scope}[->,>=stealth,semithick]
\draw[->] (b)  -- +(-90:2cm);
\end{scope}

\fill[color=black] (b) circle (.5mm);

\end{tikzpicture}

Recall from Lemma \ref{qpoly} and Remark \ref{topedegen} that, on a
$\Phi$-tope $\vv$ in the linear subspace $\lspan(\Phi)$ generated by
$\Phi$, the partition function $\F\Theta[\Phi\felarrow \Xpol]$
coincides with a quasi-polynomial $\F\qpol[\Phi\felarrow \Xpol,\vv]$
on the lattice $\lspan(\Phi)\cap \Lambda$. It is thus natural to
compare the two functions at all points of $\Lambda\cap \lspan(\Phi)$.
As we will see, the difference may be expressed as a sum of
(convolution) products of partition functions and quasi-polynomials
coming from lower-dimensional systems.

Now we can formulate Paradan's decomposition formula
(\cite{parjump}, Section 5.4, proof of Theorem 5.1) as follows.
\begin{proposition}\label{paradandecomposition}  Let $\Phi$
  be a list of vectors in $\Lambda$, and let $\Xpol\in \t$ be a
  polarizing vector for $\Phi$.
  Assume that $\gamma\in\t^*$ is such that for every $\r\in \CRF$, the
projection $\gamma_\r\in \r$ is $(\Phi\cap \r)$-regular, while the
orthogonal component
  $\yv$ is polarizing for $\Phi\setm\r$. Then
  \begin{equation}
    \label{paraeq}
    \Theta[\Phi\felarrow \Xpol]=\sum_{\r\in \CR(\Phi)}  \Theta[\Phi\setm
    \r\felarrow\ysg]\,\cdot\,
\qpol[\Phi\cap \r\felarrow \Xpol,\vv(\gamma_\r)].
\end{equation}
\end{proposition}

Observe that the set of $\gamma\in \t^*$ satisfying the assumptions of
Proposition \ref{paradandecomposition} is a complement of the union of
a finite number of hyperplanes. Indeed $\gamma_S$ is $\Phi$-regular if
it is not contained in a union of hyperplanes in $S\subset \t$, while
$\yv$ is polarizing if it is not contained in a union of
hyperplanes in $\t$.

Also note that if $-\gamma$ is in the dual cone to the cone generated by
$\Phi^+\cup -\Phi^{-}$, then all the terms but the one corresponding
to $S=\{0\}$ vanish, and hence, in this case, the identity
\eqref{paraeq} is tautological.

\begin{example} Let $\t^*=\R \alpha,\, \Lambda=\Z \alpha$,
$\Phi:=[\alpha]$ and set $\Xpol\in \t$ to be the vector satisfying
$\langle \alpha,\Xpol\rangle =1$. Then $$\Theta[\Phi\felarrow
\Xpol]=\sum_{k=0}^{\infty}e_{k\alpha}.$$

The identity
\begin{equation}\label{evident}
\sum_{k=0}^{\infty}e_{k\alpha}= \sum_{k=-\infty}^{\infty}e_{k\alpha}-\sum_{k=-\infty}^{-1}e_{k\alpha}
\end{equation}
is a particular case of Formula \eqref{paraeq}.

Indeed, in this one-dimensional case, the set $\CRF$ has two
elements: $\r=\{0\}$ and $\r=\t^*$.

If we let $\gamma=t\alpha$ for some $t>0$, then on the right hand side
of \eqref{paraeq} we have
\begin{itemize}
\item $\qpol[\Phi\felarrow \Xpol,\vv(\gamma_\r)]=\sum_{k\in
\Z}e_{k\alpha}$  for $\r=\t^*$, and \item $\Theta[\Phi\felarrow
\yv]=-\sum_{k>0}e_{- k\alpha}$, for $\r=\{0\}$.
\end{itemize}
Then formula \eqref{paraeq}  reads:
$$\Theta[\Phi\felarrow \Xpol]=\qpol[\Phi\felarrow
\Xpol,\vv(\gamma_{\r})]+\Theta[\Phi\felarrow\yv],$$
and this is Formula \eqref{evident}.

\end{example}

\begin{proof}[Proof of Proposition \ref{paradandecomposition}]

  Replacing $\gamma$ by its orthogonal projection on the subspace
  generated by $\Phi$, we may assume that $V$ is spanned by $\Phi$.
  We pass to the Fourier transforms in order to prove that the two sides of
  \eqref{paraeq} coincide. Observe that for each term on the right
  hand side of \eqref{paraeq}, the Fourier transform restricted to a
  tope of $\Phi$ is quasi-polynomial.

  We begin by showing that the Fourier coefficients of the two sides
  coincide on the tope $\vv(\gamma)$.  Indeed, the term corresponding
  to $\r=\t^*$ is $\qpol[\Phi\felarrow \Xpol,\vv(\gamma)]$, whose
  Fourier coefficients coincide with those of $\Theta[\Phi\felarrow
  \Xpol]$ on the tope $\vv(\gamma)$ by the definition of
  $\qpol[\Phi\felarrow \Xpol,\vv(\gamma)]$. On the other hand, for any
  $\r\in\CRF$ different from $\t^*$, by construction, the Fourier
  transform of the corresponding term
  $\Theta[\Phi\setm\r\felarrow\yv]\cdot \delta[\Phi\cap \r\felarrow
  \Xpol,\vv(\gamma_\r)]$ is a function on $\Lambda$ supported on the
  subset $\{\lambda;\;\langle\lambda,\yv\rangle \geq 0\}$ (cf.  Lemma
  \ref{thetaconv}).  Since
  $\langle\gamma,\yv\rangle=-|\gamma_S-\gamma|^2<0$, we see that this
  function vanishes on a conic neighborhood of the half line $\R^+
  \gamma$, and thus on $\vv(\gamma)$.

  To extend the equality of Fourier coefficients to the rest of
  $\Lambda$, we use induction on the number of elements in $\Phi$. If
  $\Phi$ is empty, then both sides are equal to 1.  Now pick an
  element $\phi\in\Phi$, and consider $\Phi'=\Phi-\{\phi\}$ (cf. the
  beginning of \S\ref{sec:fixed} for our conventions).  Clearly
  $(1-e_{\phi})\cdot\Theta[\Phi\felarrow \Xpol]=\Theta[\Phi'\felarrow
  \Xpol]$. If we restrict the Fourier transform of this equation to a
  tope $\vv$, we obtain
\[(1-e_{\phi})\,
\qpol[\Phi\felarrow \Xpol,\vv]=\qpol[\Phi'\felarrow \Xpol,\vv']
\] if $\Phi'$ generates $V$ and $\vv'$ is the tope of $\Phi'$ containing
$\vv$, while \[(1-e_{\phi})\, \qpol[\Phi\felarrow \Xpol,\vv]=0\] if
$\Phi'$ does not generate $V$.

We multiply both sides of \eqref{paraeq} by $(1-e_{\phi})$, and
compare the results. On the left hand side, we  end up with
$\Theta[\Phi'\felarrow \Xpol]$. For a term on the right hand side
corresponding to $S\in \CR(\Phi)$, we separate 3 cases:
\begin{description}
\item[1. $\phi\notin\r$] In this case, $
\r\in\CR(\Phi'),\;\Phi\cap \r=\Phi'\cap \r$ and
\[(1-e_{\phi})\,\cdot\,\Theta[\Phi\setm
\r\felarrow\yv]=\Theta[\Phi'\setm \r\felarrow\yv].
\]
Thus, after multiplication by $(1-e_{\phi})$, we end up with the term
\begin{equation}
  \label{fifi}
\Theta[\Phi'\setm
\r\felarrow {\yv}]\,\cdot\,\qpol[\Phi'\cap \r\felarrow {\Xpol},\vv(\gamma_\r)].
\end{equation}
\item[2. $\phi\in\r$, and $\r\in\CR(\Phi')$] In this case
$\Phi\setm\r=\Phi'\setm\r$ while $(\Phi\cap \r)-\{\phi\}=\Phi'\cap
\r$, which implies that
\[(1-e_{\phi})\,
\qpol[\Phi\cap \r\felarrow \Xpol,\vv(\gamma_\r)]=\qpol[\Phi'\cap
\r\felarrow \Xpol,\vv'(\gamma_\r)].
\]
 Thus we end up with the term \eqref{fifi} again.
\item [3. $\phi\in\r$, and $\r\notin\CR(\Phi')$] In this case,
\[
(1-e_{\phi})\;\qpol[\Phi\cap\r\felarrow \Xpol,\vv(\gamma_\r)]=0.\]
\end{description}

Thus multiplying the right hand side of \eqref{paraeq} by
$(1-e_{\phi})$ has the effect of replacing $\Phi$ by $\Phi'$.
 Using the
inductive assumption, we can conclude that after multiplying both
sides of \eqref{paraeq} by $(1-e_{\phi})$ for any $\phi\in\Phi$,
we obtain an identity. As $\Phi$ spans $\t^*$, this implies that
the Fourier coefficients of the difference of the two sides of
\eqref{paraeq} form a periodic function with respect to the
sublattice of finite index in $\Lambda$ generated by $\Phi$.
Since we also know that these coefficients vanish on
$\vv(\gamma)$, they must vanish on all of $\Lambda$. This
completes the proof.
\end{proof}

\section{ Decomposition of characters}

\subsection{Decomposition of a $T$-character}
From now on, we assume that the stabilizer of the action of $T$ on the
almost complex manifold $M$ is finite.

In this section, we obtain an expression (Proposition \ref{paradan})
for the character $\chie$ associated to an equivariant vector bundle
$\CE$ on $M$.

 We start with the formula \eqref{abchar} for $\chie$ from Corollary
 \ref{abcharcor}:
\[   \chie=\sum_{p\in F} \ccp\cdot \Theta[\Phi_p\felarrow \Xpol].
\]
Our plan is to substitute the decomposition formula \eqref{paraeq} for
the partition function $\Theta[\Phi_p\felarrow \Xpol]$ in each term
parametrized by $p\in F$ in this expression.  Note that while
performing this substitution, we can take a vector $\gamma^p$ in
\eqref{paraeq} depending on the fixed point $p$.  We take advantage of
this possibility: we choose a fixed vector $\gamma\in \t^*$ and we set
the vector
$$\gamma^p=\gamma-\mu(p)$$ to be the polarizing vector for the
corresponding term. Informally, this means that we expand the
denominator of the term in the fixed point formula \eqref{ab}
corresponding to $p\in F$ in the direction of $\gamma$ from $\mu(p)$.

It is clear that if we choose $\gamma$ outside a finite set of affine
hyperplanes, then $\gamma^p$ satisfies the assumptions of Proposition
\ref{paradandecomposition} for each $p\in F$. We will call such a
$\gamma$ {\em generic}. For generic $\gamma$, we obtain
\begin{equation}
  \label{subst}
\chie=\sum_{p\in F}\sum_{\r\in \CR(\Phi_p)} \ccp\cdot
\Theta[\Phi_p\setm
\r\felarrow Y_{S,\gamma^p}]\cdot\qpol[\Phi_p\cap \r\felarrow \Xpol,
\vv(\gamma^p_{\r})].
\end{equation}

Our next step is to present a geometric interpretation of this
expression. We begin by introducing certain closed subsets of $M$ with
special stabilizers.  Recall that each $X\in\t$ defines a vector field
$VX$ on $M$, which vanishes on the fixed point set $F$.

\begin{definition} \label{defsub}
  For $p\in F$ and $\r\in\CR(\Phi_p)$,  denote by $C(p,\r)$ the
  connected component of the set
  \[M^{S^\perp}=\{m\in M|\, VX(m)=0\text{ for every }X\in \r^\perp\}\]
  which contains $p$.
Let $\Sub(M)$ stand for the set of all the connected subsets
  $C(p,\r)$  of $M$ obtained this way:
\[    \Sub(M) = \{C(p,\r)|\, p\in F,\,\r\in\CR(\Phi_p)\}.\qed
\]
\end{definition}

We make two important observations:
\begin{itemize}
\item Since $M^{S^\perp}$ is also the fixed  point set of the
subtorus of $\T$ with Lie algebra $S^\perp$, the set
$C(p,\r)$ is smooth, and hence it is a submanifold of $M$.
\item For  a submanifold
$C=C(p,S)\in\Sub(M)$, the Lie algebra of the stabilizing subtorus
$T_C\subset T$ is $S^{\perp}$.
\end{itemize}

It follows then that there is a one-to-one correspondence
\[     \{(p,S)|\, p\in F,\,\r\in\CR(\Phi_p) \}\leftrightarrow
\{(p,C)|\, C\in\Sub(M),\,p\in C\cap F\},
\]
and hence we can regroup the terms of the sum in \eqref{subst}
according to the fixed point component $ C\in\Sub(M)$ to which it
corresponds.

To write down this formula, we will need to introduce some new
notation which reflects this correspondence; in particular, we will give
new names to the vectors $\gamma^p_S$ and $Y_{S,\gamma^p}$.
Using our scalar product to identify $\t$ with its dual, we can write
$\t^*=\t_C\oplus \rz$.  Recall the definition of the affine subspace
$A_C = \mu(p) + \rz\subset\t^*$ and the fact that if $p$ and $q\in
C\cap F$, then $\mu(p)-\mu(q)$ belongs to $\rz$ (cf. \eqref{muaff} and
the discussion preceding it). This implies that the projection of
$\mu(p)-\gamma$ to $\t_{\ZZ}$ does not depend on the choice of the
fixed point $p\in C\cap F.$

Using this observation, we introduce the  following notations.
\begin{definition} \label{defyg} Given $C\in\Sub(M)$ and a generic
  $\gamma$, denote by $\gamma_C$ the orthogonal projection of $\gamma$
  on the affine space $A_C$, and introduce the notation
\[ \yc \overset{\mathrm{def}}=
\gamma_C-\gamma
\]
for the polarizing vector in $\t_C$, omitting its dependence on
$\gamma$ (see Figure below).
\end{definition}

Then, given $C=C(p,S)\in\Sub(M)$, and a generic $\gamma\in\t^*$, we
have $Y_{S,\gamma^p}=Y_C$, and $\gamma^p_{\r}=\gamma_C-\mu(p)$.

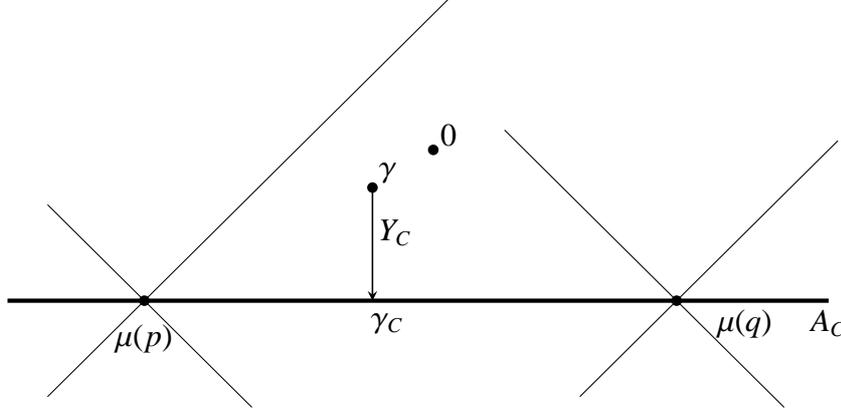
\begin{tikzpicture}
\coordinate (a) at (1,0);
\coordinate (b) at (4,1.5);
\coordinate (c) at (8,0);
\coordinate (d) at (4.8,2);

\begin{scope}[ultra thick]
\draw (a)  -- +(0:9cm);
\draw (a)  -- +(180:1.8cm);
\end{scope}

\draw (a)  -- +(-135:1.8cm);
\draw (a)  -- +(45:5.7cm);
\draw (a)  -- +(135:1.8cm);
\draw (a)  -- +(-45:2cm);

\draw (c)  -- +(-135:1.8cm);
\draw (c)  -- +(45:3cm);
\draw (c)  -- +(135:3.2cm);
\draw (c)  -- +(-45:2cm);

\draw (8.9,-.3) node {$\mu(q)$};
\draw (1,-.5) node {$\mu(p)$};
\draw (4.2,1.7) node {$\gamma$};
\draw (10,-.3) node {$A_C$};
\draw (4.2,-.3) node {$\gamma_C$};
\draw (4.3,.9) node {$\yc$};
\draw (5,2.2) node {$0$};

\begin{scope}[->,>=stealth,semithick]
\draw[->] (b)  -- +(-90:1.5cm);
\end{scope}

\fill[color=black] (a) circle (.7mm);

\fill[color=black] (b) circle (.7mm);

\fill[color=black] (c) circle (.7mm);

\fill[color=black] (d) circle (.7mm);
\end{tikzpicture}

The manifold $C$ inherits a $\T$-invariant almost complex structure
from $M$, and the set of weights of the fiber of the complex vector
bundle $\tjs C$ at $p\in C\cap F$ is $\Phi_p\cap\rz$. We can thus
regroup the terms of \eqref{subst} and
obtain the following result.
\begin{proposition} \label{resumprop}
Let $\CE$ be a complex vector bundle over an almost complex
$T$-manifold $M$, and $\gamma\in\t^*$ a generic point. Then, with the
notation introduced above, we have
\begin{equation}
  \label{substz}
  \chie=\sum_{C\in\Sub(M)} \termec,
\end{equation}
where
\begin{equation}
  \label{defdmc}
\termec:=\sum_{p\in C\cap F}
\ccp\cdot \Theta[\Phi_p\setm \rz\felarrow \yc]\cdot\delta[\Phi_p\cap
\rz\felarrow \Xpol, \vv(\gamma_{C}-\mu(p))].
\end{equation}
\end{proposition}

Our next step is to represent the contribution $\termec$ of the fixed
point set $C\in\Sub(M)$ to the sum \eqref{substz} in the
form $\loc_\mu[\tilde\CE_C,\a(\gamma_C)]$, where $\tilde\CE_C$ is a
certain infinite-dimensional bundle over $C$, and $\a(\gamma_C)$, as
usual, stands for the alcove containing $\gamma_C$.

The bundle $\tilde\CE_C$ is constructed as follows. Consider the
bundle $\NZZ=\tjs M/\tjs C$; this is a $\T$-equivariant complex
bundle\footnote{In the Kahler case, $\NZZ$ is the conormal vector
  bundle to $C$.}  on $C$, whose $\T_C$-weights are constant along
$C$.  The list $\Phi_C$ of these weights may be obtained by
restricting $\Phi_p\setm \rz$ to $\t_C$ for any $p\in C\cap F$. We
split $\Phi_C$ into two groups according to the sign of their value on
the polarizing vector $\yc\in \t_C$:
\begin{equation}
  \label{phipm}
  \Phi_C=\Phi_C^+\cup \Phi_C^-,\;\Phi_C^-=\{\phi\in\Phi_C|\,\langle\phi,\yc\rangle<0\}.
\end{equation}
This splitting induces  a direct sum decomposition of $\NZZ$:
\[ \NZZ=\NZZ_+\oplus\NZZ_-,
\] where $\NZZ_+$ and $\NZZ_-$ are the subspaces generated by
eigenvectors of $\T_C$ with weights from $\Phi_C^+$ and $\Phi_C^-$,
respectively.  Finally, define the infinite-dimensional
$\T$-equivariant virtual bundle 
 \begin{equation}
   \label{defsdot}
  \CS(\NZZ\felarrow \yc)=(-1)^{\rank \NZZ_-}
\det(\NZZ_-^*)\otimes  \bigoplus_{m=0}^{\infty}
\CS^{[m]}(\NZZ^*_-\oplus \NZZ_+)
   \end{equation}
over $C$, where $\CS^{[m]}(V)$ stands for the $m$th symmetric
tensor product of the vector space $V$, and $\det(V)$ for its top
exterior product.

  Then the
combination of the fixed point formula with
 Proposition \ref{paradandecomposition}  leads to the following statement.

\begin{proposition}\label{paradan} Let $\gamma$ be a generic
  point in $\t^*$, and denote by $\CE_C$ the restriction of $\CE$ to
  $C$.  Then for $C\in \Sub(M)$, The sum
  \begin{multline}
    \label{delec}
\loc_\mu[\CE_C\otimes
 \CS(\NZZ\felarrow  \yc),\a(\gamma_C)] \overset{\mathrm{def}}=\\
 (-1)^{\rank \NZZ_-}\sum_{m=0}^{\infty} \loc_\mu[\CE_C\otimes
\det(\NZZ_-^*)\otimes
\CS^{[m]}(\NZZ^*_-\oplus \NZZ_+), \a(\gamma_C)]    
  \end{multline}
is a well-defined formal character, and, in fact,
  \begin{equation}
    \label{Deltamuc}
\termec= \loc_\mu[\CE_C\otimes
 \CS(\NZZ\felarrow  \yc),\a(\gamma_C)], 
  \end{equation}
where the left hand side is defined in \eqref{defdmc}.

  Hence, in view of \eqref{substz}, we have the following equality in
  $\hat R(\T)$:
  \begin{equation}
    \label{thesum}
 \chi_\CE=\sum_{C\in \Sub(M)}  \loc_\mu[\CE_C\otimes
 \CS(\NZZ\felarrow  \yc),\a(\gamma_C)].
  \end{equation}
\end{proposition}
\begin{proof}
  Indeed, for $C\in \Sub(M)$, the fibers of the bundle \eqref{defsdot}
  over points of $C$ form a $\T_C$-representation with finite
  multiplicities. Recalling the definition of the formal character
  $\Theta$ from \eqref{deftheta}, we see that for $p\in C\cap F$, the
  $\T$-character of the fiber $\CS(\NZZ\felarrow \yc)_p$ is
  $\Theta[\Phi_p\setm \rz\felarrow \yc]$. Then, \eqref{Deltamuc}
  follows from comparing \eqref{defD} and \eqref{defdmc}.
\end{proof}

Proposition \ref{paradan} is a particular case of \cite[Proposition
6.14 and Formula 1.6]{par2}.  Paradan obtained this statement via
localization of the index of a transversally elliptic operator, and
then derived Proposition \ref{paradandecomposition} as a corollary of
\cite{par2}.  In our work, these statements appear in a natural order:
we proved Proposition \ref{paradandecomposition} directly for
partition functions by elementary combinatorial manipulations, and
then we deduced Proposition \ref{paradan} from the Atiyah-Bott fixed point
formula and Proposition \ref{paradandecomposition}.

\begin{remark}
Let us take a closer look at the decomposition \eqref{thesum} of the
character $\chie$.  The term corresponding to the case when $C$
consists of a single fixed point $p\in F$ is
$\ccp\cdot\Theta[\Phi_p\felarrow (\mu(p)-\gamma)]$. It is reassuring
to compare this to \eqref{abchar}, which contains a similar term:
$\ccp \cdot \Theta[\Phi_p \felarrow \Xpol]$, but where $\Phi_p$ is
reoriented with a vector $\Xpol$ independent of the point
$p$. According to Lemma \ref{thetaconv}, these two expressions,
interpreted as generalized functions on $T$, coincide with the smooth
function $\prod_{\phi\in \Phi_p}\left(1-t^{\phi}\right)^{-1}$ on the
open set $\{t\in \T;\; t^\phi\neq1\,\forall \phi\in\Phi_p\}$.
Now we observe that all the other terms of \eqref{thesum} correspond
to generalized functions supported on positive-codimensional subtori
of $T$. In particular, the term $\mcont_M[\mu,\CE,\gamma]$, which equals
$\Delta_{\mu}[\CE,\a]$ (cf. \eqref{defD}), is supported on a finite
number of points of $T$.
One can thus think of formula \eqref{thesum} as a refinement
of the Atiyah-Bott formula \eqref{ab}.  
\end{remark}

Next, we consider the supports of the Fourier transforms of the terms
of \eqref{thesum} in the Fourier dual space $\Lambda$. For simplicity, we
formulate our conclusions for the case $\CE=\CL$. (To follow the
notation, it will be helpful to consult the Figure after Definition \ref{defyg} above.)
\begin{proposition} Consider the terms of the decomposition
  \eqref{thesum} for the case $\CE=\CL$. Then the following statements
  hold.
  \label{fouriersupp}
  \begin{enumerate}
  \item Suppose that $M\neq C\in\Sub(M)$. Then the support of the
    Fourier transform $\F\mcont_C[\mu,\CL,\gamma]$, lies in
    the half-space
    \begin{equation}
      \label{halfsp}
 \{\lambda;\;
\langle\lambda,Y_C\rangle\geq \langle\gamma_C,Y_C\rangle\}.
    \end{equation}
\item When $C=M$, then the corresponding term of the sum
  \eqref{thesum} reduces to $\Delta_\mu[\CL,\a(\gamma)]$, which is a
  quasi-polynomial character.
\item On the alcove $\a(\gamma)$, the multiplicity function $\F\chie$
  coincides with the quasi-polynomial $\F\Delta_\mu[\CL,\a(\gamma)]$.
  \end{enumerate}
\end{proposition}
The first two statements immediately follow from the definition
\eqref{Deltamuc} of $\termec$. The third statement is a consequence of
the first two, since the half-spaces \eqref{halfsp} are in the
complement of $\a(\gamma)$. (Cf. Figure after Definition \ref{defyg}: the half-space
\eqref{halfsp} is the half-space under the thick line, i.e. the
one not containing $\gamma$.)  \qed

Let us verify these statements on our examples. In Example
\ref{ponex1}, the decomposition \eqref{expone}:
\begin{equation*}
\chi_{\CL^k}(t)=t^k\sum_{j=-\infty}^{\infty}t^{2j}- \sum_{j=1}^{\infty}t^{-k-2j}- \sum_{j=1}^{\infty}t^{k+2j}.
\end{equation*}
is an instance of \eqref{thesum}. The first term corresponds to
$C=\P^1$, while the other two terms come from the two fixed points.

We also give a two-dimensional example.
\begin{example}
\label{flag3}
In Example \ref{flag} (see also Example \ref{flagalcoves}), the set of
fixed point components $\Sub(M)$ consists of the following elements:
\begin{itemize}
\item  the complex 3-dimensional manifold $M$ itself,
\item the $6$ fixed points $p_w$, $w\in\Sigma_3$, corresponding to the
  vertices of the highlighted hexagon. The corresponding values of the
  moment map, are as follows:
$$\mu_{123}=4\alpha+3\beta,\,\mu_{213}=-\alpha+3\beta,\,\mu_{132}=4\alpha+\beta,$$
$$\mu_{321}=-3\alpha-4\beta,\,\mu_{231}=-3\alpha+\beta,\,\mu_{312}=-\alpha-4\beta$$

\item $9$ components isomorphic to $\P^1(\C)$, whose images are
  intervals which span the 9 lines on the picture below.  Each of
  these components contains precisely two fixed points; we will use
  the notation $C[p_v,p_w]$ for the component containing the fixed
  points $p_{v}$ and $p_w$, and $\ell\langle\mu_v,\mu_w\rangle$ for the
  corresponding line.

  For example, the fixed point component $C[p_{123},p_{213}]$ may be
  described as the set of flags of the form
\[\C v \subset \C
e_1\oplus \C e_2\subset \C e_1\oplus \C e_2\oplus \C e_3.
\] The stabilizer group of this submanifold is
$\{(t,t,u);\;t,u\in \mathrm{U}(1)\}.$
\end{itemize}

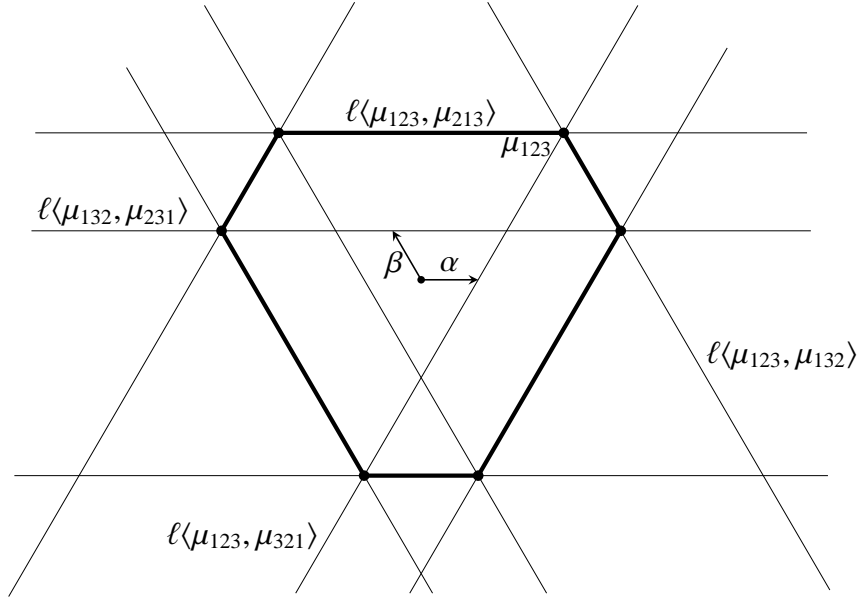
\begin{tikzpicture}
\coordinate (a) at (3,1.3);
\coordinate (b) at (4.5,1.3);
\coordinate (c) at (6.375,4.5465);
\coordinate (d) at (5.625,5.8455);
\coordinate (e) at (1.875,5.8455);
\coordinate (f) at (1.125,4.5465);

\fill[color=black] (a) circle (.7mm);
\fill[color=black] (b) circle (.7mm);
\fill[color=black] (c) circle (.7mm);
\fill[color=black] (d) circle (.7mm);
\fill[color=black] (e) circle (.7mm);
\fill[color=black] (f) circle (.7mm);

\draw (5.15,5.65) node {$\mu_{123}$};

\draw (8.5,2.9) node {$\ell\langle\mu_{123},\mu_{132}\rangle$};
\draw (3.75,6.1) node {$\ell\langle\mu_{123},\mu_{213}\rangle$};
\draw (1.4,0.5) node {$\ell\langle\mu_{123},\mu_{321}\rangle$};
\draw (-0.3,4.8) node {$\ell\langle\mu_{132},\mu_{231}\rangle$};

\draw (4.125,4.1) node {$\alpha$};
\draw (3.4,4.09) node {$\beta$};
\fill[color=black] (3.75,3.897) circle (.5mm);
\begin{scope}[->,>=stealth,semithick]
\draw[->]  (3.75,3.897)   -- +(0:.75cm);
\draw[->]  (3.75,3.897)   -- +(120:.75cm);
\end{scope}

\begin{scope}[ultra thin]
\draw (a)  -- +(-60:1.8cm);
\draw (a)  -- +(180:4.6cm);
\draw (a)  -- +(-120:1.8cm);
\draw (a)  -- +(60:7.2cm);

\draw (b)  -- +(-120:1.8cm);
\draw (b)  -- +(0:4.6cm);
\draw (b)  -- +(-60:1.8cm);
\draw (b)  -- +(120:7.2cm);

\draw (c)  -- +(60:2.8cm);
\draw (c)  -- +(-60:5.5cm);
\draw (c)  -- +(0:2.5cm);
\draw (c)  -- +(180:7.75cm);

\draw (d)  -- +(120:2cm);
\draw (d)  -- +(0:3.2cm);

\draw (e)  -- +(60:2cm);
\draw (e)  -- +(180:3.2cm);

\draw (f)  -- +(120:2.5cm);
\draw (f)  -- +(-120:5.6cm);
-- (2,1);
\end{scope}

\draw[ultra thick] (a) -- (b) -- (c) -- (d) -- (e) -- (f)
-- cycle;

\end{tikzpicture}

Let us consider $\chi_{\CL}$ as a character of the maximal torus $T$
of the adjoint group of $U(3)$ with lattice of weights
$\Lambda=\Z\alpha\oplus \Z\beta$, and set $\gamma=0$. The
decomposition \eqref{thesum} of the character $\chi_{\mathcal L}$
involves $16$ formal characters of $T$.  By symmetry with respect to
the Weyl group, we only need to describe the terms corresponding to
$M$, the term corresponding to the fixed point $\mu_{123}$, and the
terms corresponding to $C(\mu_{132},\mu_{231})$,
$C(\mu_{123},\mu_{213})$.

\begin{itemize}
\item $C=M$ contributes the polynomial character
  $\mcont_M[\mu,\CL,0]=3\sum_{\lambda\in \Lambda}e_\lambda$.
\item The term corresponding to $C=p_{123}$ is
  $$\mcont_{p_{123}}[\mu,\CL,0]=-e_{\mu_{123}}e_{(\alpha+\beta+(\alpha+\beta))}
  \cdot\sum_{k=0}^{\infty}e_{ k\alpha}\cdot\sum_{k=0}^{\infty}e_{
    k\beta} \cdot\sum_{k=0}^{\infty}e_{ k(\alpha+\beta)}$$
which is supported outside the marked hexagon.

\item
  The term corresponding to $C=C[p_{123},p_{213}]$ is
$$\mcont_{C[p_{123},p_{213}]}[\mu,\CL,0]=e_{\mu_{123}} e_{(\beta+(\alpha+\beta))} \cdot\sum_{k\in \Z}e_{ k\alpha}
\cdot\sum_{k=0}^{\infty}e_{ k\beta} \cdot\sum_{k=0}^{\infty}e_{
  k(\alpha+\beta)},$$
which is supported above the line $\ell\langle\mu_{123},\mu_{213}\rangle$.

\item
   The term corresponding to $C=C[p_{132},p_{231}]$ is
 $$\mcont_{C[p_{132},p_{231}]}[\mu,\CL,0]=-e_{\mu_{132}}  e_{(\alpha+\beta)} \cdot\sum_{k\in \Z}e_{
   k\alpha} \cdot\sum_{k=0}^{\infty}e_{ k\beta}
 \cdot\sum_{k=0}^{\infty}e_{ k(\alpha+\beta)},$$
which is supported above the line
$\ell\langle\mu_{132},\mu_{231}\rangle$.
 \end{itemize}

We can thus conclude that the multiplicities $\F\chi_{\CL}$ restricted
to the alcove $\a(0)$ (which is the small central triangle on the
picture) equals the constant 3.

\begin{remark}\label{enlarged}
  If $\gamma\in \mu(M)$ and $\CL$ is positive, then, in fact, more is
  true: $\F\chi_{\CL}$ coincides with $\F \Delta_\mu^M$ on the {\em
    closure} of the alcove $\a(\gamma)$. This effect may be observed
  in the example above. We will not use this refined
  property in this article.
\end{remark}
  \end{example}

\subsection{Decomposition of a $G$-character}
\label{sec:gchar}

Returning to the setup of \S\ref{sec:thm}, we consider a compact
connected Lie group $G$ acting compatibly on an almost complex
manifold $M$, bundles $\CE$ and $\CL$ and the connection $\nabla$ on
$\CL$. Consider the character $\chie$ of the representation of $G$ on $Q(M,\CE)$.

Recall from \S\ref{sec:thm} our notation: $\T$ is the maximal torus of
$G$, $\roots=\roots^+\cup\roots^-$ is the decomposition of the set of
roots of $G$ corresponding to the triangular decomposition
$\g_\C=\t_{\C}\oplus \n^+\oplus \n^-$. We will use $W_G$ for the Weyl
group of $G$, and $\Lambda_\dom\subset\Lambda$ will stand for the
subset of dominant weights, which serves as a fundamental domain for
the $W_G$-action on $\Lambda\subset\t^*$, and whose elements
parametrize the irreducible characters of $G$. We will
identify $\chi_\lambda$ with its restriction to $T$.

Our goal is to understand what formula \eqref{thesum} tells us about
$\chie$ as a $G$-character.
\begin{remark}
  As observed by Atiyah-Bott \cite{ati-bot2}, the Weyl character
  formula $$\chi_\lambda:=\sum_{w\in W_G}\frac{e_{
      w\lambda}}{\prod_{\alpha\in \roots^-}(1-e_{w\alpha})}$$ is the
  Atiyah-Bott fixed point formula for $\chi_{\CL_\lambda}$ associated to the line bundle
  $\CL_\lambda=G\times_{G_\lambda} \C_\lambda$ on the coadjoint orbit
  $G\lambda$.
\end{remark}

Our character $\chie\in R(G)$ may be expressed in a unique way as a
finite linear combination of irreducible characters $\chi_\lambda$,
$\lambda\in\Lambda_\dom$.  In particular, the quantity
$\int_G\chie\,dg=\dim Q(M,\CE)^G$, which we are trying to understand,
is precisely the coefficient of the trivial character in this
decomposition.  To obtain an explicit formula for this multiplicity,
we use the following simple corollary of the Weyl character formula
for $\chi_\lambda$.
\begin{lemma} We have
\begin{equation}
  \label{Weylinv}
\dim Q(M,\CE)^G = \F[\omega_G\cdot\chie](0),
\end{equation}
where
\[ \omega_G = \prod_{\alpha\in \roots^-}(1-e_{\alpha}) \in R(\T).
\]
\end{lemma}

Now we make the formal observation that multiplying $\chie$ by
$\omega_G$ amounts to tensoring $\CE$ by the trivial $\Z_2$-graded bundle over
$M$ with fiber $\wedge^\bullet\n^-=\wedge^\even\n^-\oplus\wedge^\odd\n^-$
endowed with the adjoint $\T$-action.  More precisely, let us extend
the definition of the character $\chie$ to $\Z_2$-graded vector
bundles $\mathcal{G}^\bullet=\mathcal{G}^\even\oplus\mathcal{G}^\odd$ via
\[    \chi_{\mathcal{G}^\bullet} = \chi_{\mathcal{G}^\even} -
\chi_{\mathcal{G}^\odd}.
\]
Then \eqref{Weylinv} may be written in the form
\begin{equation}
  \label{Weylinv2}
\dim Q(M,\CE)^G = \F\chi_{\CE\tensor \wedge^\bullet\n^-}(0).
\end{equation}

 Proposition \ref{paradan}  states that
   \begin{equation} \label{thesumg}
 \chi_{\CE\tensor \wedge^\bullet\n^-}=\sum_{C\in\Sub(M)}
 \loc_\mu\left[\CE_C\otimes\wedge^\bullet\n^-\otimes
 \CS(\NZZ\felarrow \yc),\a(\gamma_C)\right].
  \end{equation}

  It turns out that after tensoring $\CE$ with $\wedge^\bullet\n^-$, one can
  significantly strengthen the condition on $C$ under which the
  corresponding term in \eqref{thesumg} vanishes.

We consider the $G$-equivariant moment map $\mu_G: M\to\g^*$
satisfying equation (\ref{defmu}). Then the map $\mu$, obtained as the
composition of $\mu_G$ with the restriction $\g^*\to\t^*$, serves as a
moment map for the $\T$-action.  Note that the image $\mu_G(M)\cap \t^*$ is
usually strictly smaller than $\mu(M)$. For example, if $M=G\lambda$
is the coadjoint orbit of $\lambda\in \t^*$, then $\mu_G(M)\cap\t^*$
is the orbit $W_G\lambda$ of $\lambda$ under the Weyl group, while
$\mu(M)$ is the convex hull of $W_G\lambda$.

Also, recall the definition of the affine subspace
$A_C=\mu(p)+\t_C^\perp\subset\t^*$, where $p\in C$, associated to a
fixed point component $C\in\Sub(M)$.
\begin{theorem}\label{ParadanK}
Let $G$ be a compact connected Lie group and let $\CE$ be a
  $G$-equivariant vector bundle on the almost complex manifold $M$.
  Let $\gamma$ be generic in $\t^*$.  
Then the term
 \[
 \loc_\mu\left[\CE_C\otimes\wedge^\bullet\n^-\otimes
 \CS(\NZZ\felarrow \yc),\a(\gamma_C)\right],
 \]
 of \eqref{thesumg} vanishes if the alcove $\a(\gamma_C)$ is not
 contained in $\mu_G(C)\cap A_C\subset\t^*$.
  \end{theorem}

  This theorem is due to Paradan (\cite{par2}, Proposition 6.14 and
  Formula 1.6).  In the argument below, we will make use of Theorem
  \ref{first}, whose proof is postponed to \S\ref{sec:appendix}.

\begin{proof}
  Indeed, $\mu_G(C)$ is compact, while $\a(\gamma_C)$ is open, thus if
  $\a(\gamma_C)\not\subset\mu_G(C)\cap A_C$, then there is
  $\xi\in\a(\gamma_C)\setminus(\mu_G(C)\cap A_C)$, which is a regular
  value of $\mu$.

  According to Corollary \ref{corvanishing} proved in
  \S\ref{sec:appendix}, if we construct a $T$-equivariant isomorphism
  over $C\cap\mu^{-1}(\xi)$ between the two equivariant complex vector
  bundles with fibers $\wedge^\even\n^-$ and $\wedge^\odd\n^-$, then
  for any $T$-bundle $\G$ on $C$, we have
\[  \loc_\mu\left[\G\otimes\wedge^\bullet \n^-,\a(\gamma_C)\right] = 0.
\]

Such an isomorphism may be constructed as follows.  Let $\g=\t\oplus
\q$ be the $\T$-invariant decomposition of $\g$ with $\q=[\t,\g]$.
The dual decomposition $\g^*=\t^*\oplus \q^*$ provides us with a map
$\muperp: M\to\q^*$ satisfying
$$\mu_G(q)=\mu(q)\oplus \muperp(q).$$
The condition $\xi\notin \mu_G(C)\cap A_C$ implies that for $q\in
C\cap \mu^{-1}(\xi)$ we have $\muperp(q)\neq0$.

Fix a $G$-invariant positive definite scalar product on $\g$, and
extend it as an Hermitian product to $\g_\C$.  This induces a
$T$-invariant isomorphism $h:\q^*\to\n^-$ satisfying
$\|h(\nu)\|^2=\|\nu\|^2$.

Now recall that for a Hermitian vector space $H$, one can define a
linear map $c:H\to\End(\wedge H)$, called Clifford multiplication,
given by the formula
\begin{equation*} \label{clibefore}
c(v) :=\epsilon(v)-\epsilon(v)^*.
\end{equation*}
Here $\epsilon(v)$ is the multiplication operator in the exterior
algebra of $H$:
\begin{equation*} \label{clibefore1}
\epsilon(v) : \omega\mapsto v\wedge\omega, \,
\omega\in\wedge H,
\end{equation*}
and $\epsilon(v)^*$ is the Hermitian dual of $\epsilon(v)$, which is
the contraction  by scalar multiplication by $v$. Clearly, if $H$ is a
$T$-module with invariant Hermitian structure, then $c$ is
$T$-equivariant.

A key fact is that $c(v)^2=-\|v\|^2\cdot\mathrm{id}$, and hence $c(v)$
is a linear isomorphism whenever $v\neq 0$.  This means that the
correspondence
\[[q,\omega]\mapsto [q,c(h[\muperp(q)])\omega]\]
defines the map we sought: a $\T$-equivariant bundle-map $C\times
\wedge^{\even}\n^-\to C\times \wedge^{\odd}\n^-$, which is an
isomorphism over $\mu^{-1}(\xi)\cap C$. This completes the proof.

 \end{proof}

\section{Quasi-polynomial behavior of multiplicities: the main result}
\label{sec:quasi}



We continue with the setup of the previous section, and,
at this point, we impose the condition of positivity on our line
bundle $\CL$. Recall that this means that the curvature of the
connection $\nabla$ on $\CL$ is of the form $-i\Omega$, where the
closed 2-form $\Omega$ is such that the quadratic form
$V\mapsto\Omega_q(V,JV)$ is positive definite at each point $q\in
M$. Note that this condition, in particular, implies that $\Omega$
is symplectic.

As we pointed out in the introduction, instead, one may start by a
symplectic manifold $(M,\Omega)$ and a Kostant line bundle $\CL$, and
arrive at the same setup. Indeed, then one can choose an almost
complex structure $J$ such that the quadratic form
$V\mapsto\Omega_q(V,JV)$ is positive.  This $J$ is unique up to
continuous deformations, thus $\chiek$ does not depend on its choice.

Our purpose in this section is to analyse \eqref{thesum} for
$\CE=\CL^k$ in this situation, and prove our main result, Theorem
\ref{GS}, which we repeat here for reference.

     \begin{theorem}\label{main}
     Let $(M,J)$ be a compact, connected, almost complex manifold endowed
  with the action of a connected compact Lie group $G$, and let $\CL$
  be a positive $G$-equivariant line bundle on $M$. Suppose the set of
  fixed points under the action of the maximal torus $T$ of $G$ on $M$
  is finite.  Then
  \begin{itemize}
  \item the integer function
\[k\to \dim   Q(M,\CL^k)^G
\] is quasi-polynomial  for $k\geq 1$, and
\item this quasi-polynomial is identically zero if $0\notin \mu_G(M)$.
  \end{itemize}
\end{theorem}

\begin{proof}

  Recall some of our notations: $C\in\Sub(M)$ means that $C$ is a
  connected component of the fixed point set of a subtorus $T_C\subset
  T$ with Lie algebra $\t_C$, $A_C$ is the affine space spanned the image $\mu(C)$ of $C$ in
  $\t^*$, $\gamma_C$ is the orthogonal projection of $\gamma$ onto
  $A_C$,
  $\a(\gamma_C)$ is the alcove of $A_C$ containing $\gamma_C$
  and $\yc=\gamma_C-\gamma$ is thought of as a vector in $\t_C$
  (see Definitions \ref{defsub} and \ref{defyg}).

Combining \eqref{Weylinv2}, \eqref{thesumg} and Theorem
\ref{ParadanK}, and setting $\CE=\CL^k$, we obtain the formula

\begin{equation} \label{thesumgl} \dim Q(M,\CL^k)^G = \sum_{C}
  \F\loc_\mu\left[\CL^k\tensor\wedge^\bullet\n^-\otimes
    \CS(\NZZ\felarrow \yc),\a(\gamma_C)\right](0),
  \end{equation}
  where $\gamma$ is a generic element of $\t^*$, and the sum runs over
  $C\in\Sub(M)$ satisfying $\gamma_C\in\mu(C\cap\mu_G^{-1}(\t^*))$.

  First, consider the terms of this sum corresponding to $C\in\Sub(M)$
  for which the affine-linear subspace $A_C$ passes through the
  origin: $0\in A_C$.  For any such $C$, Lemma \ref{quasipolynomial}
  shows that $$k\to
  \F\loc_\mu\left[\CL^k\tensor\wedge^\bullet\n^-\otimes
    \CS(\NZZ\felarrow \yc),\a(\gamma_C)\right](0)$$ is a
  quasi-polynomial function of $k$.  The most important case of such a
  component is $C=M$, and the corresponding term is the
  quasi-polynomial
  $\F\loc_\mu\left[\CL^k\tensor\wedge^\bullet\n^-,\a(\gamma)\right](0)$.
  If $0$ is a regular value of $\mu$, then this is the only component
  with $0\in A_C$.

  Furthermore, the terms corresponding to $C$ with $0\in A_C$ will
  be absent in \eqref{thesumgl} if $0\notin\mu_G(M)$. Indeed, then, for $\gamma$ chosen
  sufficiently close to $0$, the orthogonal projection $\gamma_C$ of
  $\gamma$ to $A_C$ is also close to $0$, and thus
  $\gamma_C\notin\mu(C\cap\mu_G^{-1}(\t^*))\subset \mu_G(M)$.

  Now, both assertions of Theorem \ref{main} will follow if we show
  that, for $\gamma$ chosen generic and sufficiently close to $0$, the
  terms on the right hand side of \eqref{thesumgl} corresponding to
  fixed point components $C\in\Sub(M)$ with $0\notin A_C$ and
  $\gamma_C\in\mu(C\cap\mu_G^{-1}(\t^*))\subset \mu_G(M)$ vanish for
  $k\geq1$.

  Consider thus such a fixed point component $C\in\Sub(M)$ and fix a point
$q\in C\cap\mu_G^{-1}(\t^*)$ satisfying $\mu(q)=\gamma_C$. Thus we have
  \begin{equation}
    \label{qfix}
    q\in C,\, \mu_{\perp}(q)=0 \text{ and }
    \mu(q)=\gamma_C.
  \end{equation}
 Assume, {\em ad absurdum}, that the zero weight
  occurs with nonzero multiplicity in the $\T$-character
\[ \loc_\mu\left[\CL^k\otimes\wedge^\bullet\n^-\otimes
    \CS(\NZZ\felarrow \yzz),\a(\gamma_C)\right].
\]
According to Lemma \ref{simple}, this implies that the representation
of $\T_C$ on the fiber of the bundle $\CL^k\otimes\wedge^\bullet\n^-\otimes
\CS(\NZZ\felarrow \yzz)$ contains the trivial weight at any
point of $C$. In particular, the Lie algebra element $\yzz\in \t_C$
annihilates a nonzero vector in the fiber
\begin{equation}
  \label{fiber}
(\CL^k\otimes \wedge^\bullet\n^-\otimes \CS(\NZZ\felarrow \yzz))_q
\end{equation}
at our chosen point $q$.

To find a contradiction, we will give a positive lower bound on the
eigenvalues of $\yzz$ on this space assuming $\gamma\in\t^*$ is a
generic vector near $0$. Let us consider the eigenvalues of $\yzz$ on
each of the 3 tensor factors in \eqref{fiber}:

\begin{itemize}
\item The eigenvalue of $\yzz$ acting on $\CL^k_q$ is
  equal to $k\langle \mu(q),\yzz\rangle=k\langle
  \gamma_C,\yzz\rangle$.
\item The list of eigenvalues of $\yzz$ on $\wedge^\bullet\n^-$ is
  parametrized by subsets $I\subset\roots^-$ of the negative roots,
  and the eigenvalue corresponding to $I$ is $\sum_{\alpha\in I}\langle
  \alpha,\yzz\rangle$.
\item Finally, recall the definition of $\CS(\NZZ\felarrow
  \yzz)$ from  \eqref{defsdot}. Clearly, all eigenvalues of
  $\yzz$ on $\CS^{[m]}(\NZZ^*_-\oplus \NZZ_+)_q$ are
  nonnegative, and hence,
  the eigenvalues of $\yzz$ on $\CS(\NZZ\felarrow \yzz)$ are
  bounded from below by the eigenvalue of $\yzz$ on $\det
  \NZZ_-^*$. This eigenvalue equals
\begin{equation}
  \label{kceq}
 -\sum_{\eta\in\Phi_C^-}\langle\eta,\yzz\rangle,
\end{equation}
where $\Phi_C^-$ is defined in \eqref{phipm}.
\end{itemize}

The positivity of the eigenvalues of $\yzz$ on the vector space
\eqref{fiber} thus translates into the  inequality
\begin{equation}
  \label{lambdapos}
  k\langle\gamma_C,\yzz\rangle+\sum_{\alpha\in I}\langle
  \alpha,\yzz\rangle
  -\sum_{\eta\in\Phi_C^-}\langle\eta,\yzz\rangle>0\;\text{ for
  }k>0\text{ and every } I\subset\roots^-.
\end{equation}

Consider the first term: according to our assumption, the affine
subspace $A_C\subset\t^*$ does
not pass through the origin, and hence, denoting half the distance
from the origin to $A_C$ by $d_C$, we see that

\begin{equation}
  \label{evpositive}
k\langle\gamma_C,\yzz\rangle\geq d_C^2
\end{equation}
for $\gamma$ sufficiently close to the origin (see the Figure after
Definition \ref{defyg}).


%
By the definition of $\Phi_C^-$, the expression \eqref{kceq} is
also positive. Thus our worry is the set of negative contributions,
which could appear in the second term of \eqref{lambdapos}: these
correspond to those $\alpha\in\roots^-$ for which
$\langle\alpha,\yzz\rangle<0$.

Clearly, for any
$I\subset\roots^-$
\begin{equation}
  \label{gtriv}
  \sum_{\alpha\in I}\langle  \alpha,\yzz\rangle \geq
\sum_{\alpha\in \roots^-,\,\langle  \alpha,\yzz\rangle <0}\langle
\alpha,\yzz\rangle,
\end{equation}
and we have the estimates
\begin{equation}
  \label{gtoy}
   \left|\sum_{\alpha\in \roots^-,\,\langle  \alpha,\yzz\rangle <0}\langle
\alpha,\yzz\rangle - \sum_{\alpha\in \roots^-,\,(  \alpha,\gamma_C) <0}\langle
\alpha,\yzz\rangle\right| < c_1||\gamma||,
\end{equation}
and
\begin{equation}
  \label{phitoy}
  \left|
    \sum_{\eta\in\Phi_C^-}\langle\eta,\yzz\rangle-
\sum_{\eta\in\Phi_C,\,(\eta,\gamma_C)<0}\langle\eta,\yzz\rangle\right|< c_2||\gamma||
\end{equation}
 for constants $c_1,c_2$ independent of $\gamma$.

Combining inequalities \eqref{evpositive}, \eqref{gtriv},
\eqref{gtoy} and \eqref{phitoy}, we can conclude that if we prove the inequlity
\begin{equation*}
  \label{finalgamma}
d_C^2 - (c_1+c_2)||\gamma|| + \left[\sum_{\alpha\in \roots^-,\,(  \alpha,\gamma_C)<0}\langle
\alpha,\yzz\rangle-\sum_{\eta\in\Phi_C,\,(\eta,\gamma_C)<0}\langle\eta,\yzz\rangle\right]>0
\end{equation*}
for $||\gamma||$ sufficiently small, then \eqref{lambdapos} will
follow. Clearly, it is sufficient to show that the expression in the
square brackets is nonnegative, and this, in turn, will follow if we
prove that the roots $\alpha\in\roots^-$ satisfying
$(\alpha,\gamma_C)<0$ are in the list of weights $\Phi_C$ of the
action of the torus $T_C$ on the bundle $\tjs M/\tjs C$ on $C$.

This latter statement is the content of the following crucial
proposition, which, we emphasize, is the only geometric ingredient of
our proof.

We note that below, we pass from the $T_C$-weights of the bundle
$\tjs M/\tjs C$ to those of the bundle $\TT^JM$, which has the effect of
reversing all signs, and adding a number of zero-weights.

\begin{proposition} \label{crucial} Let $(M,\Omega,\mu_G)$ be a Hamiltonian
  $G$-manifold, and let $J$ be a $G$-invariant  almost complex structure such that
  $\Omega(v,Jv)>0$ for all tangent vectors $v\neq0$.  Fix a point
  $q\in M$ such that $\mu_G(q)=\mu(q)$, i.e. $\muperp(q)=0$. Then the
  list of complex weights of the stabilizer group $\T_q$ on $\TT^J_q M$
  with respect to the almost complex structure $J$ contains the following sublist
  of restricted roots:
   \begin{equation}
     \label{posweights}
[\alpha|\t_q;\;\alpha\in\roots,\,(\mu(q),\alpha)>0].
   \end{equation}
\end{proposition}

\begin{proof}
  Recall that $VX(q)$ stands for the tangent vector in $\TT_qM$
  corresponding to $X\in\g$ under the $G$-action on $M$. As our
  calculations below will take place in the tangent space $\TT_qM$, we
  will omit the dependence on $q$ from our notation.  

  We need to show that under the conditions described above, there is
  a nonzero tangent vector $W\in\TT_qM$ such that
\[     X\cdot W = \langle\alpha,X\rangle\,J(W)
\]
for every $X\in\t_q$. Here, $X\cdot W$ stands for the action of the
stabilizer Lie algebra $\g_q$ on $\TT_qM$.

Let us extend the map $V:\g\to \T_qM$  to $\g_\C$ by complex linearity via
\[  V[X+iY] = VX+J(VY).
\]

Then $V:\g_\C\to (T_qM,J)$ is a map of complex vector spaces, which is
equivariant with respect to the action of $T_q$, the stabilizer group of $q$,
acting on $\g_\C$ by the adjoint action, and on $\T_qM$ by its natural
action.

Let $\alpha\in\roots$ be a root satisfying $(\mu(q),\alpha)>0$, and
let $X_\alpha,Y_\alpha\in\g$ be two Lie algebra elements, such that
$X_\alpha+iY_\alpha$ is in the root space
$\g_\C(\alpha)$,
and the triple
\[ E_\alpha=X_\alpha+iY_\alpha\;F_\alpha=-X_\alpha+iY_\alpha,\;
H_\alpha=2i[X_\alpha,Y_\alpha]
\] satisfies the commutation relations of the standard basis of
$\mathfrak{sl}_2$:
\[ [H_\alpha,E_\alpha]=2E_\alpha, [H_\alpha,F_\alpha]=-2F_\alpha,
[E_\alpha,F_\alpha]=H_\alpha.
\]
Then we also have
$(\beta,\alpha)=c_\alpha\langle\beta,iH_\alpha\rangle$ with
$c_\alpha>0$ for any $\beta\in\t^*$.

For any $X\in\t_q$, we have $X\cdot VE_\alpha =
\langle\alpha,X\rangle\, J(VE_\alpha)$. Thus the statement is proved
if we verify that $VE_\alpha$ does not vanish. To show this, we
prove that $\Omega_q(VE_\alpha,J(VE_\alpha))\neq0$.  Indeed, we have
$$
\Omega_q(VE_\alpha,J(VE_\alpha))=\Omega_q(VX_\alpha,J(VX_\alpha))+
\Omega_q(VY_\alpha,J(VY_\alpha))-2\Omega_q(VX_\alpha,VY_\alpha).
 $$
 The first two terms of this sum are nonnegative by our assumptions on
 $\Omega$. As for the last term, from the key identity \eqref{musat},
 we have
\[   \Omega_q(VX_\alpha,VY_\alpha) + \langle \nabla_{VY_\alpha}\, \mu_G,X_\alpha\rangle=0,
\]
where $\nabla$ denotes the directional derivative. On the other hand,
from the invariance of $\mu_G$, we have
$\langle \nabla_{VY_\alpha}\, \mu_G,X_\alpha\rangle+\langle\mu_G,
[X_\alpha,Y_\alpha]\rangle=0$, which leads to
\begin{equation}\label{oms}
\Omega_q(VX_\alpha,VY_\alpha)=\langle\mu_G(q),[X_\alpha,Y_\alpha]\rangle=
\langle \mu_G(q), -iH_\alpha/2\rangle = 
\frac{-1}{2c_\alpha}(\mu_G(q),\alpha)<0.
\end{equation}
This completes the proof of Proposition \ref{crucial} and the proof of
our main result Theorem \ref{main} as well.
\end{proof}
\end{proof}

\section{The asymptotic result in the torus case}
\label{sec:appendix}

The purpose of this section is to give a concise proof of the
following variant of Theorem \ref{first}, which is a special case of
the asymptotic result proved by Meinrenken in \cite{mei1}.
\begin{theorem}\label{vanishing}
  Let $M$ be a compact almost complex $\T$-manifold, let $\CL$ be a
  $\T$-equivariant line bundle over $M$ with moment map $\mu$, and let
  $\CE^\bullet=\CE^\even\oplus\CE^\odd$ be a $\Z_2$-graded equivariant vector
  bundle over $M$. Assume that for a compact subset $\b$
  of the regular values of $\mu$,   $\CE^\even$ and $\CE^\odd$
  are equivariantly isomorphic on $\mu^{-1}(\gamma)$ for every $\gamma\in\b$.
Then there is a  $K>0$ such that
  \[
\F\chi_{\CE^\bullet\otimes\CL^k}(\lambda)=0 \text{ for }k>K \text{ and }\lambda
  \in k\b\cap\Lambda.
\]
\end{theorem}

\begin{proof}


Again, our starting point is  the Atiyah-Bott fixed point formula \eqref{ab}:
\begin{equation}
  \label{ab1}
  \chi_{\CE^\bullet\otimes\CL^k} = \sum_{p\in F} \frac{e_{k\mu(p)}\ccpb}{\prod_{\phi\in
      \Phi_p}(1-e_{\phi})}.
\end{equation}
Here we used the notation $\ccpb = \tau_p[\CE^\even] -\tau_p[\CE^\odd]$.

It clearly follows from our hypothesis that if $p\in F$ is such that $\mu(p)\in\b$, we
have $\ccpb=0$. Thus, introducing the subset
$F'=\{p\in F;\;\mu(p)\notin\b\}$ of all fixed points, we can
write
\begin{equation}
  \label{fcharint}
  \F\chi_{\CE^\bullet\otimes\CL^k}(\lambda)=\int_\T e_{-\lambda}(t)\sum_{p\in F'}
 \frac{e_{k\mu(p)}(t)\ccpb(t)}{\prod_{\phi\in \Phi_p}
(1-e_{\phi}(t))}\,dt.
\end{equation}
 To estimate this integral, we would like to exchange the
summation and the integration in this formula. However, the terms
of the sum are singular expressions, and thus we can only estimate
the part of this integral where the terms of the sum are bounded.

To find this partial estimate, we proceed as follows. Consider the
open set
\[  \T_\reg=\{g\in \T| \; e_{\phi}(g)\neq1\,\forall \phi\in\Phi_p,\,p\in F'\},
\]
of those elements $g\in \T$ for which the terms of our sum are
regular, and for each $g\in \T_\reg$ pick a ball $U_g\subset\t$
centered at $0\in\t$ such that $g\exp(U_g)\subset \T_\reg$.
Now, let $\rho_g:\T\to[0,1]$ be an auxiliary smooth function with
compact support on $g\exp(U_g)$, and consider the piece
\begin{equation}
  \label{piece}
  \int_\T\rho_g(t)e_{-\lambda}(t)\chi_{\CE^\bullet\otimes\CL^k}(t)\,dt
\end{equation}
of the integral in \eqref{fcharint} supported in $g\exp(U_g)$.
Pulling this integral back to $\t$ via the map $g\exp:\t\to \T$, we
can estimate the absolute value of \eqref{piece} as being less or
equal than
\begin{equation}
  \label{fixint}
  \sum_{p\in F'}\left|\int_\t e^{ik\langle\mu(p)-\lambda/k,X\rangle}\,
\frac{\rho_g(g\exp(X))\,\ccpb(g\exp(X))}
{\prod_{\phi\in \Phi_p}(1- e^{i\langle\phi,X\rangle}e_\phi(g))}\;dX\right|.
\end{equation}
Note that we omitted the constant factor
$e^{ik\mu(p)-i\lambda}(g)$, since it is of absolute value 1.

Now we recall the following standard estimate from Fourier analysis.
\begin{lemma}\label{Fourier} Let $0\neq\eta\in\t^*$, and $H:\t\to\C$
  be a smooth compactly supported
  function. Then for every positive integer $d$,  the inequality
\[
\left|\int_\t e^{i\langle\eta,X\rangle} H(X)\;
dX\right|\leq\frac{C_d(H)}{\|\eta\|^{2d}}
\]
holds, where the constant $C_d(H)$  depends only on a finite number of
derivatives of $H$; in fact, one can take
\[ C_d(H)= \max_{X\in \t}\left|\left[\sum_i \partial_i^2\right]^{d}(H(X))\right|.
\]

\end{lemma}

Now we return to \eqref{fixint}, and consider expression
$\mu(p)-\lambda/k$ in the exponent. Since, according to our
assumptions, $\lambda/k\in \b$, and $\mu(p)$ is not in $\b$, we have
the bound $|\mu(p)-\lambda/k|\geq \delta$ for some positive
$\delta$. Applying Lemma \ref{Fourier} to our integrand with
$\eta=k(\mu(p)-\lambda/k)$, we obtain the following
\begin{corollary} For $g\in \T_\reg$, and smooth function $\rho_g:
  \T\to[0,1]$ with compact support in $U_g$, the integral
  \eqref{piece}, goes to zero faster
  than any negative power of $k$, uniformly for $\lambda\in k\b$.
\end{corollary}


In order to bound the rest of the integral \eqref{fcharint}, for
each $g\in \T\setminus \T_\reg$, we will replace the Atiyah-Bott
formula by an expression, which is regular at $g$. Such formulas
were given in \cite{ber-ver85}; here we sketch the setup and the
relevant notions. We begin with the case of the unit element of $\T$:
$g=\unit$. We follow the exposition  of (\cite{ber-get-ver}, chapters 7,8).

For a manifold $M$ with a $\T$-action,
we define the algebra $\CA_\T(M)$ of {\em equivariant forms}
as the space of smooth maps $\alpha:\t\to \Gamma(\WT)^\T$,
from $\t$ to the set of invariant differential forms on $M$.
As a matter of notation, we will write $\alpha(X)$ for the resulting
differential form on $M$,
and $\alpha(X,q)$ for the value of this differential form at $q\in M$.

The equivariant differential $D:\CA_\T(M)\to\CA_\T(M)$, given by the formula
$$D\alpha(X)=d\alpha(X)-\iota(VX)\alpha(X),$$
satisfies  $D^2=0$.
(Here $\iota(v)\alpha$ is the contraction of the differential form
$\alpha$ by the vector $v$.)
Accordingly, $\alpha\in\CA_\T(M)$ is called {\em
equivariantly closed} if $D\alpha=0$. The formulas in
\cite{ber-ver82} express the integral $\int_M\alpha:\t\to\C$ of an
equivariantly closed form $\alpha$ in terms of local data on $M$.

Returning to our setup of a $\T$-manifold $M$, endowed with a line
bundle $\CL$ with curvature $R_{\CL}=-i\Omega$, we observe that we
have already encountered such equivariantly closed forms: indeed,
equation \eqref{musat} may be interpreted as saying that the
expression
\begin{equation}
  \label{lcurv}
R_\CL(X)=R_{\CL}+L_X-\nabla_{VX}=i\langle \mu,X\rangle-i\Omega,
\end{equation}
the {\em equivariant curvature} of the bundle $\CL$, is equivariantly
closed. The equivariant curvature may be constructed for any equivariant
bundle $\CB$ over $M$ by choosing a $\T$-invariant connection
$\nabla$ on $\CB$ with curvature $R_\CB$.  Then, again, we can define
$R_\CB(X)=R_\CB+L_X-\nabla_{VX}$ which is a smooth map from $\t$ to
the $\T$-invariant sections of the bundle of algebras
$\WT\tensor\End(\CB)$.  We can then define the equivariant forms
\begin{equation}
  \label{defeqforms}
\ch_\CB(X) = \Tr_\CB\left[\exp(R_\CB(X))\right],  \quad
\Todd_\CB(X)=
\mathrm{det}_\CB\left[\frac{R_\CB(X)}{1-\exp(-R_\CB(X))}\right],
\end{equation}
where the trace and the determinant are taken in $\End(\CB)$.
These forms are called, respectively, the {\em equivariant Chern
class} and the {\em equivariant Todd class} of the bundle $\CB$. Note
that the latter is only defined in a neighborhood of $0\in\t$.

Now let us denote by $\bba$ the set of those regular values $\xi$ of $\mu$ in $\t^*$ 
for which $\CE^\even$ and $\CE^\odd$ are isomorphic
over $\mu^{-1}(\xi)$; clearly, $\bba$ is an open set containing $\b$.

Observe that since $\CE^\even$ is isomorphic to $\CE^\odd$ over
$\mu^{-1}(\bba)$, we can assume that the corresponding connections
$\nabla_{\CE^\even}$ and $\nabla_{\CE^\odd}$ are chosen to coincide
over $\mu^{-1}(\bba)$.

Applying the above construction to the bundles $\CE^\even$, $\CE^\odd$ and
$\tj M$, we obtain the equivariant curvature forms
$R_{\CE^\even}$, $R_{\CE^\odd}$ and $R_{\tj M}$, respectively,  and thus
we have \begin{equation}
  \label{curvesame}
  R_{\CE^\even}(X,q) =R_{\CE^\odd}(X,q) \quad\text{if }\mu(q)\in\bba.
\end{equation}

Now we are ready to write down the relevant formula from
\cite{ber-ver85} (see also \cite{ber-get-ver}, chapter 8):
$$\chi_{\CE^\bullet\otimes \CL^k}(\exp X)=\frac{1}{(2\pi i)^{\dim M/2}}
\int_{M} \ch_{\CL^k}(X) \left[\ch_{\CE^\even}(X)-\ch_{\CE^\odd}(X)\right]
\Todd_{\tj M}(X);$$ this equality is valid for $X$ from the neighborhood
$U_\unit$ of $0\in\t$ where $\Todd_{\tj M}(X)$ is defined.

 Writing $\ch_{\CE^\bullet}(X)$ for
 $\ch_{\CE^\even}(X)-\ch_{\CE^\odd}(X)$
 and using \eqref{lcurv}, we can rewrite this expression as 
\begin{equation}
  \label{idrewrite}
\chi_{\CE^\bullet\otimes \CL^k}(\exp X)= \frac{1}{(2\pi i)^{\dim
M/2}}\int_{M}e^{i k\langle \mu,X\rangle -ik\Omega} \ch_{\CE^\bullet}(X)
\Todd_{\tj M}(X).
\end{equation}

Now we proceed similarly to our analysis of the Atiyah-Bott formula
above. We choose an auxiliary smooth function $\rho_\unit:\T\to[0,1]$
with compact support in $\exp(U_\unit)$, and we write
\begin{multline}
 \label{idmult}
 (2i\pi)^{\dim
M/2} \int_\T\rho_\unit(t)e_{-\lambda}(t)\chi_{\CE^\bullet\otimes\CL^k}(t)\,dt=\\
\int_M \sum_{j=0}^{\dim M/2}
\frac{(-ik)^j}{j!}\Omega^j
 \times \int_\t \rho_\unit(\exp(X)) e^{i k\langle
\mu(q)-\lambda/k,X\rangle} \ch_{\CE^\bullet}(X) \Todd_{\tj M}(X) \; dX.
\end{multline}
According to \eqref{curvesame}, the factor
$\ch_{\CE^\bullet}(X,q)$ vanishes whenever $\mu(q)\in \bba$. Denoting
the distance between $\b$ and the complement of $\bba$ by $\delta$, we
can again assume that $|\mu(q)-\lambda/k|>\delta$ whenever
$\lambda\in k\b$. Since both $M$ and the support of
$\rho_\unit$ are compact, we have bounds on the derivatives of the
integrand in \eqref{idmult}, which are uniform in $q$. Hence we can
apply Lemma \ref{Fourier} again to conclude that for each $d$, there
is a constant $C_d$, independent of $q$, such that the integral over
$\t$ in \eqref{idmult} is bound by $C_dk^{-2d}$. Integrating over $M$
then gives us
\begin{corollary}
  The integral \eqref{idmult} goes to zero as $k\to\infty$ faster than
  any negative power of $k$ as $k\to\infty$.
\end{corollary}

Finally, we can extend these arguments to all $g\in \T$, using the
generalization of \eqref{idrewrite} given in \cite [Theorem
3.23]{ber-ver85}.  We first introduce the twisted versions of our
characteristic forms: if $s\in \T$ acts trivially on $M$, then we can
define the twisted Chern character
\[   \ch_{\CB,s}(X)  = \Tr\left[s\exp(R_\CB(X))\right],
\]
and
\[  D_{\CB,s} = \det\left[1-s^{-1}\exp(-R_\CB(X))\right],
\]
as $s$ acts fiberwise in any $\T$-equivariant vector bundle over
$M$.

Now let $g\in \T$ be an arbitrary element, denote by $M^g$ the
submanifold fixed by $g$ (thus $g$ acts trivially on $M^g$) and
let $NM^g$ be the normal bundle of $M^g$ in $M$. Then the formula
in \cite{ber-ver85} states that
\begin{equation}\label{berlinevergne}
\chi_{\CE^\bullet\tensor\CL^k}(g \exp X)=\frac1{(2\pi i)^{\dim M^g/2}}
\int_{M^g} \frac{\ch_{\CL^k,g}(X)\ch_{\CE^\bullet,g}(X)
\Todd_{M^g}(X)}{D_{NM^g,g}(X)}
\end{equation}
for $X$ in a neighborhood $U_g$ of $0$.

From here on, the arguments are identical to those we gave in the case
$g=\unit$, and hence they will be omitted. The result may be
formulated as follows.
\begin{lemma}
  \label{summ}
  For $g\in \T$, let $U_g$ be a neighborhood of $0\in\t$ such that for
  $X\in U_g$ the characteristic classes $\Todd_{M^g}(X)$ and
  $D_{NM^g,g}(X)^{-1}$ are defined on $M^g$. Then for
  any smooth function $\rho_g:\T\to[0,1]$ compactly supported in
  $g\exp(U_g)$, and any $\lambda\in k\b$, the integral
\[\int_\T\rho_g(t)e_{-\lambda}(t)\chi_{\CE^\bullet\otimes\CL^k}(t)\,dt\]
goes to zero faster than any negative power of $k$.
\end{lemma}
Now we can easily finish the proof of the theorem.
 Indeed, the sets
$\{g\exp(U_g)|\; g\in \T\}$
form an open cover of the compact torus $\T$.
We can thus pick a finite subset $S\subset \T$ such that
$\cup_{g\in S}g\exp(U_g)=\T$.
Next, we choose a partition of unity subordinated to this cover,
i.e functions $\rho_g:\T\to[0,1]$, $g\in S$
such that $\rho_g$ is compactly supported in $g\exp(U_g)$ and
 $\sum_{g\in S}\rho_g=1$. Then, for $\lambda\in k\b$, we have
\[
\int_\T e_{-\lambda}(t)\chi_{\CE^\bullet\otimes\CL^k}(t)\,dt =
\sum_{g\in S}\int_\T\rho_g(t)e_{-\lambda}(t)\chi_{\CE^\bullet\otimes\CL^k}(t)\,dt.
\]
Each term of the sum goes to zero as $k\to\infty$ uniformly in
$\lambda$, and hence so does their sum, the expression on the left
hand side, which equals $\F\chi_{\CE^\bullet\tensor\CL^k}$. As
$\F\chi_{\CE^\bullet\tensor\CL^k}$ is an integer,  this completes
the proof of Theorem \ref{vanishing}.

\end{proof}\bigskip

Finally, we can formulate an important corollary of  Theorem
\ref{vanishing}, which is used in the paper.  Recall Definition
\ref{defDelta} and Proposition \ref{Dchi}.

 \begin{corollary}\label{corvanishing}
   Let $\CE^\bullet=\CE^\even\oplus\CE^\odd$ be a $\Z_2$-graded bundle over
   $M$, and let $\a\subset \t^*$ be an alcove.  If for some $\gamma\in
   \a$, which is a regular
   value of $\mu$, the $\T$-equivariant bundles $\CE^\even$ and $\CE^\odd$ are
   isomorphic on $\mu^{-1}(\gamma)$, then
   $\loc_{\mu}[\CE^\bullet,\a]=0.$ In particular, if $\mu^{-1}(\gamma)$ is empty, then
$\loc_{\mu}[\CE^\bullet,\a]=0$ for any $\Z_2$-graded vector bundle $\CE^{\bullet}$.

\end{corollary}

\begin{proof}
 According to Lemma \ref{qpoly} and Proposition \ref{Dchi}, this
 follows from the fact that for a compact $\b\subset\a$ and $k$ sufficiently large
\[  \lambda\in k\b\quad\Rightarrow\quad\F\chi_{\CE^\bullet\otimes\CL^k}(\lambda)=0.
\]
\end{proof}

\section{List of notations}
\label{lon}

\begin{itemize}
\item $(M,\omega)$ -- compact symplectic manifold; $\CE$ --
  vector, $\CL$ -- line bundle over  $M$.
\item $\CE_C$ -- vector bundle restricted to the submanifold $C$.
\item $\TT M$ -- the tangent bundle of $M$, $J\in\mathrm{End}(\TT M)$ stands for
  an almost complex structure, $\tj$ and $\tjs$ denote the $\pm i$ eigenspaces
  of $J$.
\item $\T$ -- compact torus group, $\t$ -- its Lie algebra, $\Lambda$
  -- weight lattice of $\T$, $G$ -- compact Lie group with maximal torus
  $\T$ and Lie algebra $\g$.
\item $\mu_G:M\to \g^*$ and $\mu:M\to\t^*$ -- moment maps,
  corresponding to a not necessarily positive line bundle.
\item $F$ stands for the $T$-fixed point set of $M$, which we assume
  to be finite. For $p\in M$, we denote by $\Phi_p$ the list of
  tangent weights of $M$ at $p$, and by $\Psi_p$ the list of
  $T$-weights of $\CE_p$; the weight of $\CL_p$ equals $\mu(p)$. We
  will use the notation $\tau_p[\CE]=\sum_{\lambda\in\Psi_p}e_\lambda$.
\item $\F\eta$ -- the Fourier transform/multiplicity function of the
  formal character $\eta$ of $\T$.
\item $\Theta[\Phi\felarrow \Xpol]$ -- formal character associated with
  the list of weights $\Phi$ and oriented by the vector $\Xpol$ (cf. \eqref{deftheta}).
\item $\delta[\Phi\felarrow \Xpol,\vv]$ -- formal quasi-polynomial
  character, whose multiplicity function coincides with that of
  $\Theta[\Phi\felarrow \Xpol]$ on the tope $\vv$ (cf. Lemma \ref{qpoly}).
\item $\loc_\mu[\CE,\a]$ -- the asymptotic character associated to
  $\CE$ and $\mu$ (cf. Definition \ref{defDelta}).
\item $G_S,\g_S$ -- connected component of generic stabilizer group of
  the subset $S$ of a $G$-space, and its Lie algebra. In particular,
  $T_C$ and $\t_C$ stand for the connected component of the generic
  stabilizer group of the subset $C\subset M$ under the action of the
  maximal torus.
\item $\CRF$ -- set of linear subspaces spanned by subsets of the list
  $\Phi$; for $\r\in\CRF$ and $\gamma\in\t^*$, denote by $\gamma_S$
  the projection of $\gamma$ onto $\r$ and by $Y_{S,\gamma}$ the
  vector in $\t$ corresponding to $\gamma_S-\gamma$ under the the
  isomorphism $\t\cong\t^*$ (cf diagram after Remark \ref{remrs}).
\item $\Sub(M)$ -- set of connected components of fixed point sets of
  $M$ with respect to the actions of a subtorus group of the maximal
  torus $T$ (Definition \ref{defsub}).
\item For $C\in\Sub(M)$, denote by $A_C$ the affine subspace
  $\mu(p)+\t_C^{\perp}\subset\t^*$, where $p\in C\cap F$ (cf. \eqref{muaff});
  for $\gamma\in\t^*$, let $\gamma_C$ be the projection of $\gamma$
  onto $A_C$, and let $Y_C\in\t$ be the vector corresponding to
  $\gamma_C-\gamma$ (cf. diagram after Definition
  \ref{defyg}). Finally, we denote by $\termec$ the contribution of
  $C$ to the expression of $\chie$ in Proposition \ref{resumprop}.
\end{itemize}

\end{document}